\documentclass[11pt,reqno]{amsart}
\usepackage[top=1.2in, bottom=1.2in, left=1.2in, right=1.2in]{geometry}
\usepackage{amsfonts, amssymb, amsmath, enumerate, bm, xspace, graphicx, caption, subcaption}
\usepackage[linktoc=page, colorlinks, linkcolor=blue, citecolor=blue]{hyperref}
\usepackage{pbox}
\usepackage{epstopdf}

\numberwithin{equation}{section}

\DeclareMathOperator{\E}{\mathbb{E}}
\DeclareMathOperator{\diag}{diag}

\renewcommand{\Pr}[2][]{\mathbb{P}_{#1} \left\{ #2 \rule{0mm}{3mm}\right\}}

\def \CC {\mathcal{C}}

\def \EE {\mathcal{E}}

\def \LL {\mathcal{L}}

\def \NN {\mathcal{N}}

\def \RR {\mathcal{R}}

\def \a {\alpha}

\def \e {\varepsilon}
\def \d {\delta}
\def \l {\lambda}

\def \tran {\mathsf{T}}

\def \one {{\textbf 1}}

\newtheorem{theorem}{Theorem}[section]

\newtheorem{corollary}[theorem]{Corollary}
\newtheorem{lemma}[theorem]{Lemma}

\theoremstyle{remark}
\newtheorem{remark}[theorem]{Remark}

\title[]{Concentration and regularization of random graphs}

\author{Can M. Le, Elizaveta Levina \and Roman Vershynin}

\address{Department of Statistics, University of California, Davis, One Shields Ave, Davis, CA 95616, U.S.A.}
\email{canle@ucdavis.edu}

\address{Department of Statistics, University of Michigan, 1085 S. University Ave, Ann Arbor, MI 48109, U.S.A.}
\email{elevina@umich.edu}

\address{Department of Mathematics, University of Michigan, 530 Church St, Ann Arbor, MI 48109, U.S.A.}
\email{romanv@umich.edu}

\thanks{E. L. is partially supported by NSF grants DMS-1159005 and DMS-1521551.   R. V. is partially supported by NSF grant 1265782 and U.S. Air Force grant FA9550-14-1-0009. This work was done while C. L. was a Ph.D. student at the University of Michigan.}

\date{\today}

\begin{document}


\begin{abstract}
  This paper studies how close random graphs are typically to their expectations.
  We interpret this question through the concentration of the adjacency and Laplacian matrices
  in the spectral norm.
  We study inhomogeneous Erd\"os-R\'enyi random graphs on $n$ vertices, where edges form
  independently and possibly with different probabilities $p_{ij}$.
  Sparse random graphs whose expected degrees are $o(\log n)$ fail to concentrate;
  the obstruction is caused by vertices with abnormally high and low degrees.
  We show that concentration can be restored if we regularize the degrees of such vertices,
  and one can do this in various ways.
  As an example, let us reweight or remove enough edges to make all degrees bounded above by $O(d)$ where
  $d=\max np_{ij}$.
  Then we show that the resulting adjacency matrix $A'$ concentrates with the optimal rate:
  $\|A' - \E A\| = O(\sqrt{d})$.
  Similarly, if we make all degrees bounded below by $d$ by adding weight $d/n$ to all edges,
  then the resulting Laplacian concentrates with the optimal rate: $\|\LL(A') - \LL(\E A')\| = O(1/\sqrt{d})$.
  Our approach is based on Grothendieck-Pietsch factorization, using which we construct a new
  decomposition of random graphs.
  We illustrate the concentration results with an application to the community detection problem in
  the analysis of networks.
\end{abstract}

\maketitle


\section{Introduction}

Many classical and modern results in probability theory, starting from the Law of Large Numbers,
can be expressed as concentration of random objects about their expectations.
The objects studied most are
sums of independent random variables, martingales, nice functions on product probability spaces and
metric measure spaces.
For a panoramic exposition of concentration phenomena in modern probability theory and related fields,
the reader is referred to the books \cite{ledoux2001,Boucheron&Lugosi&Massart2013}.

This paper studies concentration properties of random graphs.
The first step of such study should be to decide how to interpret the statement that a
{\em random graph $G$ concentrates near its expectation}. To do this, it will be useful to look at the
graph $G$ through the lens of the matrices classically associated with $G$,
namely the adjacency and Laplacian matrices.

Let us first build the theory for the adjacency matrix $A$; the Laplacian will be discussed in
Section~\ref{s: Laplacian}. We may say that $G$ concentrates
about its expectation if $A$ is close to its expectation $\E A$ in some natural matrix norm;
we interpret the expectation of $G$ as the weighted graph with adjacency matrix $\E A$.
Various matrix norms could be of interest here. In this paper, we study concentration in the spectral norm.
This automatically gives us a tight control of all eigenvalues and eigenvectors, according to
Weyl's and Davis-Kahan perturbation inequalities (see \cite[Sections III.2 and VII.3]{Bhatia1996}).		

Concentration of random graphs interpreted this way, and also of general random matrices,
has been studied in several communities,
in particular in random matrix theory, combinatorics and network science.

We will study random graphs generated from an
{\em inhomogeneous Erd\"os-R\'enyi model} $G(n, (p_{ij}))$,
where edges are formed independently with given probabilities $p_{ij}$, see \cite{Bollobas2007}.
This is a generalization of the classical Erd\"os-R\'enyi model $G(n,p)$ where all edge probabilities
$p_{ij}$ equal $p$. Many popular graph models arise as special cases of $G(n, (p_{ij}))$,
such as the stochastic block model, a benchmark model in the analysis of networks
\cite{Holland83} discussed in Section~\ref{s: networks}, and random subgraphs of given graphs.

Often, the question of interest is estimating some features of the probability matrix $(p_{ij})$
from random graphs drawn from $G(n, (p_{ij}))$.   Concentration of
adjacency matrix and Laplacian matrix around their expectations, when
it holds,
guarantees that such features can be recovered.    As an example of
this use of our concentration results, we will show that if $(p_{ij})$
has a block structure, the blocks can be accurately estimated from a single realization of $G(n, (p_{ij}))$ even when the average vertex degree is bounded.

\subsection{Dense graphs concentrate}				\label{s: dense}

The cleanest concentration results are available for the classical Erd\"os-R\'enyi model $G(n,p)$
in the {\em dense} regime. In terms of the expected degree $d = pn$, we have with high probability that
\begin{equation}         \label{eq: concentration dense ER}
\|A - \E A\| = 2 \sqrt{d} \, (1+o(1)) \quad \text{if} \quad d \gg \log^4 n,
\end{equation}
see \cite{FurKom80, Vu2007, Lu&Peng2013}. Since $\|\E A\| = d$, we see that the typical deviation
here behaves like the square root of the
magnitude of expectation -- just like in many other classical results of probability theory.
In other words, {\em dense random graphs concentrate well}.

The lower bound on density in \eqref{eq: concentration dense ER} can be essentially relaxed all the way down
to $d = \Omega(\log n)$. Thus, with high probability we have
\begin{equation}         \label{eq: concentration dense}
\|A - \E A\| = O(\sqrt{d}) \quad \text{if} \quad d = \Omega(\log n).
\end{equation}
This result was proved in \cite{FeiOfe05} based on the method developed by J.~Kahn and E.~Szemeredi \cite{Friedman&Kahn&Szemeredi1989}.
More generally, \eqref{eq: concentration dense} holds for any inhomogeneous Erd\"os-R\'enyi model $G(n, (p_{ij}))$
with maximal expected degree $d = \max_i \sum_j p_{ij}$. This generalization
can be deduced from a recent result of S.~Bandeira and R.~van~Handel \cite[Corollary~3.6]{Bandeira&Handel2014},
while a weaker bound $O(\sqrt{d \log n})$ follows from
concentration inequalities for sums of independent random matrices \cite{Oliveira2010}.
Alternatively, an argument in \cite{FeiOfe05} can be used to prove \eqref{eq: concentration dense}
for a somewhat larger but still useful value
\begin{equation}         \label{eq: d}
d = \max_{ij} n p_{ij} ,
\end{equation}
see \cite{Lei&Rinaldo2013, Chin&Rao&Vu2015}.
The same can be obtained by using Seginer's bound on random matrices \cite{Hajek&Wu&Xu2014}.
As we will see shortly, our paper provides an alternative and completely different approach to
general concentration results like \eqref{eq: concentration dense}.

\subsection{Sparse graphs do not concentrate}

In the {\em sparse} regime, where the expected degree $d$ is bounded, concentration breaks down.
According to \cite{Krivelevich&Sudakov2003}, a random graph from $G(n,p)$ satisfies with high probability that
\begin{equation}         \label{eq: no-concentration sparse ER}
\|A\| = (1+o(1)) \sqrt{d(A)} = (1+o(1)) \sqrt{\frac{\log n}{\log \log n}}
\quad \text{if} \quad d = O(1),
\end{equation}
where $d(A)$ denotes the maximal degree of the graph (a random quantity).
So in this regime we have $\|A\| \gg \|\E A\| = d$, which shows that {\em sparse random graphs do not concentrate}.

\medskip

What exactly makes the norm $A$ abnormally large in the sparse regime? The answer is,
the vertices with too high degrees.
In the dense case where $d \gg \log n$,
all vertices typically have approximately the same degrees $(1+o(1))d$.
This no longer happens in the sparser regime $d \ll \log n$;
the degrees do not cluster tightly about the same value anymore.
There are vertices with too high degrees; they are captured by the second
inequality in \eqref{eq: no-concentration sparse ER}.
Even a single high-degree vertex can blow up the norm of the adjacency matrix.
Indeed, since the norm of $A$ is bounded below by the Euclidean norm of each of its rows,
we have $\|A\| \ge \sqrt{d(A)}$.

\subsection{Regularization enforces concentration}

If high-degree vertices destroy concentration, can we ``tame'' these vertices?
One proposal would be to remove these vertices from the graph altogether.
U.~Feige and E.~Ofek \cite{FeiOfe05} showed that this works for $G(n,p)$ --
{\em the removal of the high degree vertices enforces concentration}.
Indeed, if we drop all vertices with degrees, say, larger than $2d$, the
the remaining part of the graph satisfies
\begin{equation}							\label{eq: Feige-Ofek}
\|A' - \E A'\| = O(\sqrt{d})
\end{equation}
with high probability, where $A'$ denotes the adjacency matrix of the new graph.
The argument in \cite{FeiOfe05} is based on the method developed by J.~Kahn and E.~Szemeredi \cite{Friedman&Kahn&Szemeredi1989}. It extends to the inhomogeneous Erd\"os-R\'enyi model $G(n, (p_{ij}))$
with $d$ defined in \eqref{eq: d}, see \cite{Lei&Rinaldo2013, Chin&Rao&Vu2015}.
As we will see, our paper provides an alternative and completely different approach to such results.

Although the removal of high degree vertices solves the concentration problem,
such solution is not ideal, since those vertices are in some sense the most important ones.
In real-world networks, the vertices with highest degrees are ``hubs'' that hold the network together.
Their removal would cause the network to break down into disconnected components, which leads
to a considerable loss of structural information.

Would it be possible to regularize the graph in a more gentle way -- instead of removing
the high-degree vertices, reduce the weights of their edges just enough to keep the degrees
bounded by $O(d)$? The main result of our paper states that this is true.
Let us first state this result informally; Theorem~\ref{thm: main formal} provides a more general
and formal statement.

\begin{theorem}[Concentration of regularized adjacency matrices]  \label{thm: main informal}
  Consider a random graph from the inhomogeneous Erd\"os-R\'enyi model $G(n, (p_{ij}))$,
  and let $d = \max_{ij} n p_{ij}$.
  For all high degree vertices of the graph (say, those with degrees larger than $2d$),
  reduce the weights of the edges incident to them in an arbitrary way, but so that all degrees
  of the new (weighted) graph become bounded by $2d$.
  Then, with high probability, the adjacency matrix $A'$ of the new graph concentrates:
  $$
  \|A' - \E A\| = O(\sqrt{d}).
  $$
  Moreover, instead of requiring that the degrees become bounded by $2d$,
  we can require that the $\ell_2$ norms of the rows of the new adjacency matrix
  become bounded by $\sqrt{2d}$.
\end{theorem}

\subsection{Examples of graph regularization}			\label{s: partial cases}

The regularization procedure in Theorem~\ref{thm: main informal} is very flexible. Depending
on how one chooses the weights, one can obtain as partial cases several results we
summarized earlier, as well as some new ones.

\begin{enumerate}[1.]

\item {\em Do not do anything to the graph.}
In the dense regime where $d = \Omega(\log n)$, all degrees are already bounded
by $2d$ with high probability. This means that the original graph satisfies $\|A - \E A\| = O(\sqrt{d})$.
Thus we recover the result of U.~Feige and E.~Ofek \eqref{eq: concentration dense},
which states that {\em dense random graphs concentrate well}.

\item {\em Remove all high degree vertices.}
If we remove all vertices with degrees larger than $2d$, we recover
another result of U.~Feige and E.~Ofek \eqref{eq: Feige-Ofek},
which states that {\em the removal of the high degree vertices enforces concentration}.

\item {\em Remove just enough edges from high-degree vertices.}
Instead of removing the high-degree vertices with all of their edges,
we can remove just enough edges to make all degrees bounded by $2d$.
This milder regularization still produces the concentration bound \eqref{eq: Feige-Ofek}.

\item \label{item: excess} {\em Reduce the weight of edges proportionally to the excess of degrees.}
Instead of removing edges, we can reduce the weight of the existing edges,
a procedure which better preserves the structure of the graph. For instance,
we can assign weight $\sqrt{\l_i \l_j}$ to the edge between vertices $i$ and $j$,
choosing $\l_i := \min(2d/d_i, 1)$ where $d_i$ is the degree of vertex $i$.
One can check that this makes the $\ell_2$ norms of all rows of the adjacency matrix bounded by $2d$.
By Theorem~\ref{thm: main informal}, such regularization procedure leads to
the same concentration bound \eqref{eq: Feige-Ofek}.

\end{enumerate}

\subsection{Concentration of Laplacian}						\label{s: Laplacian}

So far, we have looked at random graphs through the lens of their adjacency matrices.
A different matrix that captures the geometry of a graph is the (symmmetric, normalized) Laplacian matrix,
defined as
\begin{equation}         \label{eq: Laplacian}
\LL(A) = D^{-1/2} (D - A) D^{-1/2} = I - D^{-1/2} A D^{-1/2}.
\end{equation}
Here $I$ is the identity matrix and $D = \diag(d_i)$ is the
diagonal matrix with degrees $d_i = \sum_{j=1}^n A_{ij}$ on the diagonal.
The reader is referred to \cite{ChungFan1997} for an introduction to graph Laplacians
and their role in spectral graph theory. Here we mention just two basic facts:
the spectrum of $\LL(A)$
is a subset of $[0,2]$, and the smallest eigenvalue is always zero.

Concentration of Laplacians of random graphs has  been studied in
\cite{Oliveira2010,Chaudhuri&Chung&Tsiatas2012,Qin&Rohe2013,Joseph&Yu2013,Gao&Ma&Zhang&Zhou2015}.
Just like the adjacency matrix, the Laplacian is known to concentrate in the dense regime where $d = \Omega(\log n)$,
and it fails to concentrate in the sparse regime. However, the obstructions to concentration are opposite.
For the adjacency matrices, as we mentioned, the trouble is caused by high-degree vertices. For the Laplacian,
the problem lies with {\em low-degree vertices}. In particular, for $d = o(\log n)$ the graph is likely to have
isolated vertices; they produce multiple zero eigenvalues of $\LL(A)$, which are easily seen to
destroy the concentration.

In analogy to our discussion of adjacency matrices,
we can try to regularize the graph to ``tame'' the low-degree vertices in various ways,
for example remove the low-degree vertices, connect them to some other vertices,
artificially increase the degrees $d_i$ in the definition \eqref{eq: Laplacian}
of Laplacian, and so on. Here we will focus on the following simple way of regularization proposed in
\cite{amini2013pseudo} and analyzed in \cite{Joseph&Yu2013,Gao&Ma&Zhang&Zhou2015}.
Choose $\tau > 0$ and add the same number $\tau/n$ to all
entries of the adjacency matrix $A$, thereby replacing it with
$$
A_\tau := A + (\tau/n) \one \one^\tran
$$
in the definition \eqref{eq: Laplacian} of the Laplacian.
This regularization raises all degrees $d_i$ to $d_i + \tau$. If we choose $\tau \sim d$,
the regularized graph does not have low-degree vertices anymore.

The following consequence of Theorem~\ref{thm: main informal} shows
that such regularization indeed forces Laplacian to concentrate. Here we state this result informally;
Theorem~\ref{thm: Laplacian formal} provides a more formal statement.

\begin{theorem}[Concentration of the regularized Laplacian]  \label{thm: Laplacian informal}
  Consider a random graph from the inhomogeneous Erd\"os-R\'enyi model $G(n,(p_{ij}))$,
  and let $d = \max_{ij} n p_{ij}$.
  Choose a number $\tau \sim d$.
  Then, with high probability, the regularized Laplacian $\LL(A_\tau)$ concentrates:
  $$
  \|\LL(A_\tau) - \LL(\E A_\tau)\| = O \Big( \frac{1}{\sqrt{d}} \Big).
  $$
\end{theorem}

We will deduce this result from Theorem~\ref{thm: main informal} in Section~\ref{s: Laplacian proof}.
Theorem~\ref{thm: Laplacian informal} is an improvement upon a bound in \cite{Gao&Ma&Zhang&Zhou2015} that had an extra $\log d$ factor,
and it was conjectured there that the logarithmic factor is not needed.
Theorem~\ref{thm: Laplacian informal} confirms this conjecture.

\subsection{A numerical experiment}\label{sec: numerical}

To conclude our discussion of various ways to regularize sparse graphs,
let us illustrate the effect of regularization by a numerical experiment.
Consider an inhomogeneous Erd\"os-R\'enyi graph with $n=1000$ vertices,
$90\%$ of which have expected degrees $7$ and  $10\%$ percent have expected degrees $35$.
We then regularize the graph by reducing the weights of edges proportionally to the excess of degrees --
just as we described in Section~\ref{s: partial cases} item~\ref{item: excess},
except that we use the overall average degree (approximately $10$) instead of $d$ (which
results in a more severe weight reduction suitable for our illustration purpose).

Figure~\ref{fig: spectral shrinkage} shows the histogram of the spectrum of $A$ (left) and $A'$ (right). As we can see,
the high degree vertices lead to the long tails in the histogram of
the eigenvalues, and regularization shrinks these tails toward the
bulk.

\begin{figure}[htp]			
  \centering
  \begin{subfigure}{0.44\textwidth}
    \includegraphics[width=\textwidth]{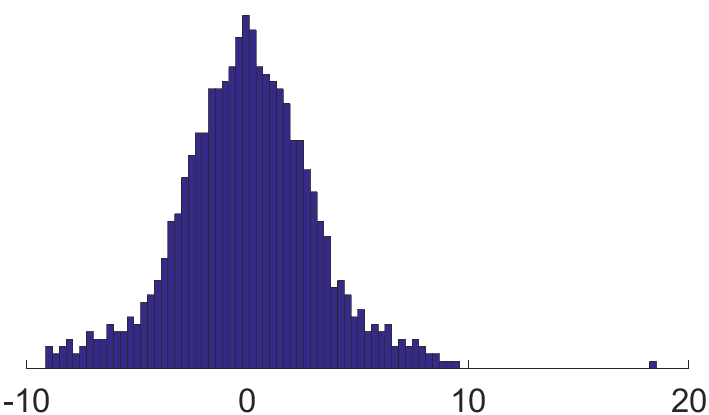}
  \end{subfigure}
  \qquad \qquad
  \begin{subfigure}{0.44\textwidth}
    \includegraphics[width=\textwidth]{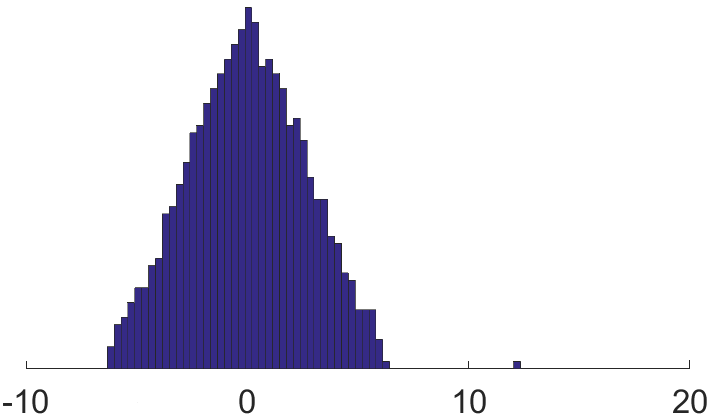}
  \end{subfigure}
  \caption{Histogram of the spectrum of adjacency matrix $A$ (left) and regularized adjacency matrix $A'$ (right)
  for a sparse random graph generated from the inhomogeneous Erd\"os-R\'enyi model with $n=1000$ vertices,
$90\%$ of which have expected degrees $7$ and  $10\%$ percent have expected degrees $35$. }
  \label{fig: spectral shrinkage}
\end{figure}

\subsection{Application: community detection in networks}				\label{s: networks}

Concentration of random graphs has an important application to
statistical analysis of networks, in particular to the problem of
community detection.    A common way of modeling communities in
networks is the {\em stochastic block model} \cite{Holland83},
which is a special case of the inhomogeneous Erd\"os-R\'enyi model considered in this paper.
For the purpose of this example, we focus on the simplest version of the
stochastic block model $G(n,\frac{a}{n},\frac{b}{n})$,  also known as
the balanced planted partition model,  defined as follows.
The set of vertices is divided into two subsets (communities) of size $n/2$ each.
Edges between vertices are drawn independently with probability $a/n$ if they are in the same
community and with probability $b/n$ otherwise.

The community detection problem is to recover the community labels of vertices from a single realization of the random graph model.
A large literature exists on both the recovery algorithms and the theory establishing when a recovery is possible
\cite{Decelle.et.al.2011,Mossel&Neeman&Sly2014, Mossel&Neeman&Sly2014a,
  Mossel&Neeman&SlyOnConsistencyThresholds2014,
  Abbe&Bandeira&Hall2014, Massoulie:2014,Bordenave.et.al2015non-backtracking}.
There are methods that perform {\em better than a random guess} (i.e. the fraction of misclassified vertices is bounded away from $0.5$ as $n \to \infty$ with high probability) under the condition
$$
(a-b)^2 > 2 (a+b),
$$
and no method can perform better than a random guess if this condition is violated.

Moreover, {\em strong consistency}, or exact recovery (labeling {\em all} vertices correctly with high probability) is possible when the expected degree $(a+b)/2$ is of order $\log n$ or larger and $a$ and $b$ are sufficiently separated, see \cite{Mossel&Neeman&SlyOnConsistencyThresholds2014,McS01,Bickel&Chen2009,Hajek&Wu&Xu2014,Cai&Li2015}.
{\em Weak consistency} (the fraction of mislabeled vertices going to 0 with high probability) is achievable if and only if
$$
(a-b)^2 > C_n (a+b) \quad \text{with} \quad C_n \rightarrow \infty,
$$
see \cite{Mossel&Neeman&SlyOnConsistencyThresholds2014}. Many of these results hold in the non-asymptotic
regime, for graphs of fixed size $n$. Thus, for any $\e>0$ there exists $C_\e$ (which only depends on $\e$) such that
one can recover communities up to $\e n$ mislabeled vertices as long as
$$
(a-b)^2 > C_\e (a+b).
$$
In particular, recovery of communities is possible even for very sparse graphs -- those with bounded expected degrees.
Several types of algorithms are known to succeed in this regume, including non-backtracking walks
\cite{Mossel&Neeman&Sly2014,Massoulie:2014,Bordenave.et.al2015non-backtracking},
spectral methods \cite{Chin&Rao&Vu2015} and methods based on semidefinite programming 
\cite{Guedon&Vershynin2014,Montanari&Sen2015}.

As an application of the new concentration results, we show that the {\em regularized spectral clustering} \cite{amini2013pseudo,Joseph&Yu2013},
one of the simplest most popular algorithms for community detection, can
recover communities in the sparse regime.  In general, spectral
clustering works by computing the leading eigenvectors of either the
adjacency matrix or the Laplacian or their regularized versions, and
running the $k$-means clustering algorithm on these eigenvectors to
recover the node labels.    In the simple case of the model
$G(n,\frac{a}{n},\frac{b}{n})$, one can simply assign nodes to
communities based on the sign (positive or negative) of the corresponding entries of the
eigenvector $v_2(\LL(A_\tau))$ corresponding to the second smallest
eigenvalue of regularized Laplacian matrix $\LL(A_\tau)$ (or the
regularized adjacency matrix $A'$).

Let us briefly explain how our concentration results validate regularized spectral clustering.
If the concentration of random graphs holds and $\LL(A_\tau)$ is close to $\LL(\E A_\tau)$,
then the standard perturbation theory (Davis-Kahan theorem below) shows that $v_2(\LL(A_\tau))$ is close to
$v_2(\LL(\E A_\tau))$, and in particular, the signs of these two eigenvectors must agree on most vertices.
An easy calculation shows that the signs of $v_2(\LL(\E A_\tau))$ recover the communities exactly:
this vector is a positive constant on one community and a negative constant on the
other. Therefore, the signs of $v_2(\LL(A_\tau))$ must recover the communities up to a small fraction
of misclassified vertices.

Before stating our result, let us quote a simple version of the Davis-Kahan theorem perturbation
theorem (see e.g. \cite[Theorem VII.3.2]{Bhatia1996}).

\begin{theorem}[Davis-Kahan theorem]\label{thm: Davis-Kahan}
Let $X,Y$ be symmetric matrices such that the second smallest eigenvalues of $X$ and $Y$ have multiplicity one and they are of distance at least $\delta>0$ from the remaining eigenvalues of $X$ and $Y$.
Denote by $x$ and $y$ the eigenvectors of $X$ and $Y$ corresponding to the second largest eigenvalues of $X$ and $Y$, respectively. Then
$$
\min_{\beta=\pm 1}\|x+\beta y\| \le \frac{2\|X-Y\|}{\delta}.
$$
\end{theorem}

\begin{corollary}[Community detection in sparse graphs]\label{cor: community detection}
Let $\e>0$ and $r\ge 1$. Let $A$ be the adjacency matrix drawn from the stochastic block model $G(n,\frac{a}{n},\frac{b}{n})$. Assume that
\begin{equation}\label{eq: ab}
  (a-b)^2>C_\e(a+b)
\end{equation}
where $C_\e = Cr^4\e^{-2}$ and $C$ is an appropriately large absolute constant.
Choose $\tau$ to be the average degree of the graph, i.e. $\tau=(d_1+\cdots+d_n)/n$ where $d_i$ are vertex degrees.
Then with probability at least $1-e^{-r}$, we have
$$
\min_{\beta = \pm 1}\|v_2(\LL(A_\tau))+\beta v_2(\LL(\E A_\tau))\| \le \e.
$$
In particular, the signs of the entires of $v_2(\LL(A_\tau))$ correctly estimate the partition into the two communities, up to at most $\e n$ misclassified vertices.
\end{corollary}

\begin{proof}
We apply Theorem~\ref{thm: Davis-Kahan} with $X=\LL(A_\tau)$ and $Y=\LL(\E A_\tau)$.
A simple calculation shows that the spectral gap $\delta$ defined in
Theorem~\ref{thm: Davis-Kahan} is of the order $(a-b)/(a+b)$.
The claim of Corollary~\ref{cor: community detection} then follows from Davis-Kahan Theorem~\ref{thm: Davis-Kahan},
Concentration Theorem~\ref{thm: Laplacian formal} (which is a formal version of Theorem~\ref{thm: Laplacian informal})
and condition \eqref{eq: ab}.
\end{proof}

\subsection{Organization of the paper}

In Section~\ref{s: decomposition}, we state a formal version of Theorem~\ref{thm: main informal}.
We show there how to deduce this result from a new decomposition of random graphs,
which we state as Theorem~\ref{thm: decomposition}.
We prove this decomposition theorem in Section~\ref{s: decomposition proof}.
In Section~\ref{s: Laplacian proof}, we state and prove a formal version of Theorem~\ref{thm: Laplacian informal}
about the concentration of the Laplacian. We conclude the paper with Section~\ref{s: questions}
where we propose some questions for further investigation.

\subsection*{Acknowledgement}

The authors are grateful to Ramon van Handel for several insightful comments on the
preliminary version of this paper.

\section{Full version of Theorem~\ref{thm: main informal}, and reduction to a graph decomposition} \label{s: decomposition}

In this section we state a more general and quantitative version of Theorem~\ref{thm: main informal},
and we reduce it to a new form of graph decomposition, which can be of interest on its own.

\begin{theorem}[Concentration of regularized adjacency matrices]  \label{thm: main formal}
  Consider a random graph from the inhomogeneous Erd\"os-R\'enyi model $G(n,(p_{ij}))$,
  and let $d = \max_{ij} n p_{ij}$.
  For any $r \ge 1$, the following holds with probability at least $1-n^{-r}$.
  Consider any subset consisting of at most $10n/d$ vertices, and
  reduce the weights of the edges incident to those vertices in an arbitrary way.
  Let $d'$ be the maximal degree of the resulting graph.
  Then the adjacency matrix $A'$ of the new (weighted) graph satisfies
  $$
  \|A' - \E A\| \le C r^{3/2} \big( \sqrt{d} + \sqrt{d'} \big).
  $$
  Moreover, the same bound holds for $d'$ being the maximal $\ell_2$ norm of the
  rows of $A'$.
\end{theorem}

In this result and in the rest of the paper, $C, C_1, C_2, \ldots$ denote absolute constants
whose values may be different from line to line.

\begin{remark}[Theorem~\ref{thm: main formal} implies Theorem~\ref{thm: main informal}]
    The subset of $10n/d$ vertices in Theorem~\ref{thm: main formal} can be completely arbitrary.
    So let us choose the high-degree vertices, say those with degrees larger than $2d$.
    There are at most $10n/d$ such vertices with high probability; this follows by
    an easy calculation, and also from Lemma~\ref{lem: degrees of subgraphs}.
    Thus we immediately deduce Theorem~\ref{thm: main informal}.
\end{remark}

\begin{remark}[Tight upper bound]
If we do not reduce weights of any edges and $d$ is bounded, then the
upper bound in Theorem~\ref{thm: main formal} is tight (up to a
constant depending on $r$). This is because of \eqref{eq: no-concentration
  sparse ER}, which states that the adjacency matrix does
not concentrate in the sparse regime without regularization.
\end{remark}

\begin{remark}[Method to prove Theorem~\ref{thm: main formal}]
One may wonder if Theorem~\ref{thm: main formal} can be proved by developing
an $\epsilon$-net argument similar to the method of J.~Kahn and E.~Szemeredi \cite{Friedman&Kahn&Szemeredi1989} and its versions \cite{Alon&Kahale1997, FeiOfe05, Lei&Rinaldo2013, Chin&Rao&Vu2015}.
Although we can not rule out such possibility, we do not see how this method could handle
a general regularization. The reader familiar with the method can easily
notice an obstacle. The contribution of the so-called light couples becomes hard to
control when one changes, and even reduces, the individual entries of $A$ (the weights of edges).

We will develop an alternative and somewhat simpler approach, which will be able to handle a general
regularization of random graphs.
It sheds light on the specific structure of graphs that enables concentration. We are going to identify this structure through a
{\em graph decomposition} in the next section. But let us pause briefly to mention the following
useful reduction.
\end{remark}

\begin{remark}[Reduction to directed graphs]
  Our arguments will be more convenient to carry out if the adjacency matrix $A$ has all independent entries.
  To be able to make this assumption, we can decompose $A$ into the upper-triangular and a lower-triangular parts,
  both of which have independent entries. If we can show that each of these parts concentrate about its expectation,
  it would follow that $A$ concentrate about $\E A$ by triangle inequality.

  In other words, we may prove Theorem~\ref{thm: main formal} for {\em directed} inhomogeneous
  Erd\"os-R\'enyi graphs, where edges between any vertices and in any direction appear
  indepednently with probabilities $p_{ij}$. In the rest of the argument, we will only work with such random
  directed graphs.
\end{remark}

\subsection{Graph decomposition}

In this section, we reduce Theorem~\ref{thm: main formal} to the following decomposition
of inhomogeneous Erd\"os-R\'enyi directed random graphs.
This decomposition may have an independent interest. Throughout the paper, we
denote by $B_\NN$ the matrix which coincides with a matrix $B$ on a subset of edges $\NN \subset [n] \times [n]$
and has zero entries elsewhere.

\begin{theorem}[Graph decomposition]				\label{thm: decomposition}
  Consider a random directed graph from the inhomogeneous Erd\"os-R\'enyi model,
  and let $d$ be as in \eqref{eq: d}.
  For any $r \ge 1$, the following holds with probability at least $1-3n^{-r}$.
  One can decompose the set of edges $[n] \times [n]$ into three classes $\NN$, $\RR$ and $\CC$
  so that the following properties are satisfied for the adjacency matrix $A$.
  \begin{itemize}
    \item The graph concentrates on $\NN$, namely
    $\|(A - \E A)_\NN\| \le C r^{3/2} \sqrt{d}$.
    \item Each row of $A_\RR$ and each column of $A_\CC$ has at most $32r$ ones.
  \end{itemize}
  Moreover, $\RR$ intersects at most $n/d$ columns,
  and $\CC$ intersects at most $n/d$ rows of $[n] \times [n]$.
\end{theorem}

Figure~\ref{fig: decomposition} illustrates a possible decomposition Theorem~\ref{thm: decomposition} can provide.
The edges in $\NN$ form a big ``core'' where the graph concentrates well
even without regularization. The edges in $\RR$ and $\CC$ can be thought of (at least heuristically)
as those attached to high-degree vertices.

\begin{figure}[htp]			  	
  \centering
  \includegraphics[width=0.2\textwidth]{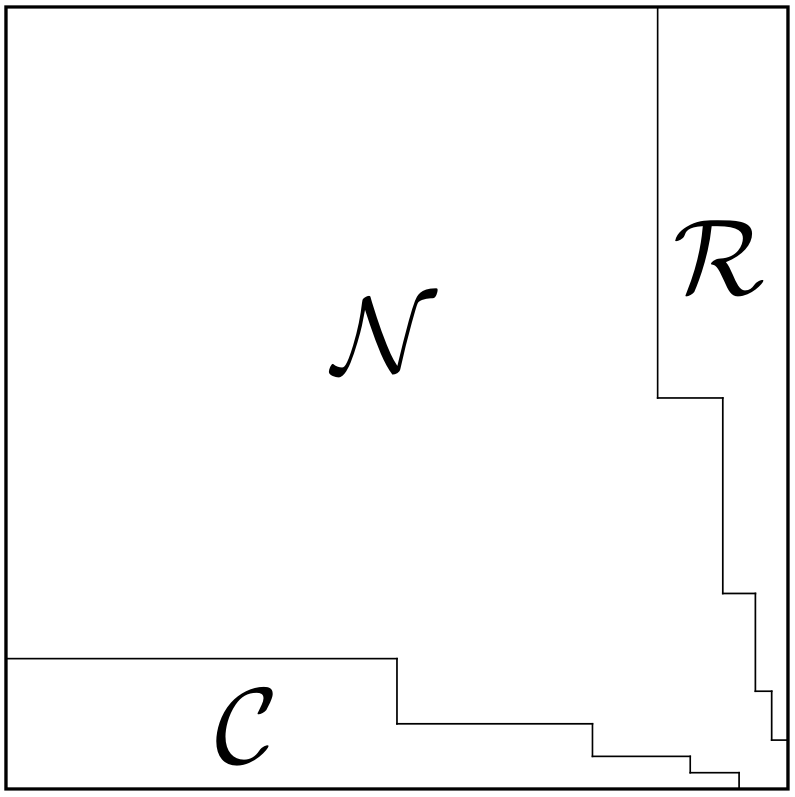}
  \caption{An example of graph decomposition in Theorem~\ref{thm: decomposition}.}
   \label{fig: decomposition}	
\end{figure}


We will prove Theorem~\ref{thm: decomposition} in Section~\ref{s: decomposition proof};
let us pause to deduce Theorem~\ref{thm: main formal} from it.

\subsection{Deduction of Theorem~\ref{thm: main formal}}

First, let us explain informally how the graph decomposition could lead to Theorem~\ref{thm: main formal}.
The regularization of the graph does not destroy the properties of $\NN$, $\RR$ and $\CC$
in Theorem~\ref{thm: decomposition}. Moreover,
regularization creates a new property for us, allowing for a good control of the {\em columns}
of $\RR$ and {\em rows} of $\CC$.
Let us focus on $A_{\RR}$ to be specific.
The $\ell_1$ norms of all columns of this matrix are at most $d'$, and the $\ell_1$
norms of all rows are $O(r)$ by Theorem~\ref{thm: decomposition}.
By a simple calculation which we will do in Lemma~\ref{lem: norm bound with L1 condition}, this implies
that $\|A_\RR\| = O(\sqrt{rd'})$.
A similar bound can be proved for $\CC$. Combining $\NN$, $\RR$ and $\CC$ together will lead to
the error bound $O(r^{3/2}(\sqrt{d} + \sqrt{d'}))$ in Theorem~\ref{thm: main formal}.

\medskip

To make this argument rigorous, let us start with the simple calculation we just mentioned.

\begin{lemma}				\label{lem: norm bound with L1 condition}
  Consider a matrix $B$ in which each row has $\ell_1$ norm at most $a$,
  and each column has $\ell_1$ norm at most $b$.
  Then $\|B\| \le \sqrt{ab}$.
\end{lemma}

\begin{proof}
Let $x$ be a vector with $\|x\|_2=1$.
Using Cauchy-Schwarz inequality and the assumptions, we have
\begin{eqnarray*}
 \| B x \|_2^2 &=& \sum_i \Big( \sum_j B_{ij} x_j \Big)^2
 \le \sum_i \Big( \sum_j |B_{ij}| \sum_j |B_{ij}| x_j^2 \Big)\\
 &\le& \sum_i \Big( a \sum_j |B_{ij}| x_j^2 \Big)
 = a \sum_j x_j^2 \sum_i |B_{ij}|
\le  a \sum_j x_j^2 b = ab.
\end{eqnarray*}
Since $x$ is arbitrary, this completes the proof.
\end{proof}
\begin{remark}[Riesz-Thorin interpolation theorem implies Lemma~\ref{lem: norm bound with L1 condition}]
Lemma~\ref{lem: norm bound with L1 condition} can also be deduced from Riesz-Thorin interpolation theorem (see e.g. \cite[Theorem~2.1]{Stein&Shakarchi2011}),
since the maximal $\ell_1$ norm of columns is the $\ell_1 \rightarrow \ell_1$ operator norm, and the maximal $\ell_1$ norm of rows is the $\ell_\infty \rightarrow \ell_\infty$ operator norm.
\end{remark}

We are ready to formally deduce the main part of Theorem~\ref{thm: main formal}
from Theorem~\ref{thm: decomposition}; we defer the ``moreover'' part to Section~\ref{s: moreover}.

\begin{proof}[Proof of Theorem~\ref{thm: main formal} (main part).]
Fix a realization of the random graph that satisfies the conclusion of Theorem~\ref{thm: decomposition},
and decompose the deviation $A' - \E A$ as follows:
\begin{equation}         \label{eq: decomposition A'-EA}
A' - \E A = (A' - \E A)_\NN + (A' - \E A)_\RR + (A' - \E A)_\CC.
\end{equation}
We will bound the spectral norm  of each of the three terms separately.

\medskip

{\bf Step 1. Deviation on $\NN$.} Let us further decompose
\begin{equation}         \label{eq: on CC decomposed}
(A' - \E A)_\NN = (A - \E A)_\NN - (A-A')_\NN.
\end{equation}
By Theorem~\ref{thm: decomposition}, $\|(A - \E A)_\NN\| \le C r^{3/2} \sqrt{d}$.
To control the second term in \eqref{eq: on CC decomposed},
denote by $\EE \subset [n] \times [n]$ the subset of edges that
are reweighed in the regularization process. 
Since $A$ and $A'$ are equal on $\EE^c$, we have
\begin{align}
\|(A-A')_\NN\|
  &= \|(A - A')_{\NN \cap \EE}\|
  \le \|A_{\NN \cap \EE}\| \qquad \text{(since $0 \le A-A' \le A$ entrywise)} \nonumber\\
  &\le \|(A-\E A)_{\NN \cap \EE}\| + \|\E A_{\NN \cap \EE}\|  \qquad \text{(by triangle inequality).} \label{eq: NN on EE}
\end{align}
Further, a simple restriction property implies that
\begin{equation}         \label{eq: NN cap EE}
\|(A-\E A)_{\NN \cap \EE}\| \le 2 \|(A-\E A)_\NN\|.
\end{equation}
Indeed, restricting a matrix onto a product subset of $[n] \times [n]$ can only reduce its norm.
Although the set of reweighted edges $\EE$ is not a product subset, it can be decomposed into two product subsets:
\begin{equation}         \label{eq: EE}
\EE = \big( I \times [n] \big) \cup \big( I^c \times I \big)
\end{equation}
where $I$ is the subset of vertices incident to the edges in $\EE$.
Then \eqref{eq: NN cap EE} holds;
right hand side of that inequality is bounded by $C r^{3/2} \sqrt{d}$
by Theorem~\ref{thm: decomposition}.
Thus we handled the first term in \eqref{eq: NN on EE}.

To bound the second term in \eqref{eq: NN on EE}, we can use another restriction property
that states that the norm of the matrix with non-negative entries can only reduce by
restricting onto any subset of $[n] \times [n]$ (whether a product subset or not). This yields
\begin{equation}         \label{eq: EA on NN EE}
\|\E A_{\NN \cap \EE}\| \le \|\E A_\EE\| \le \|\E A_{I \times [n]}\| + \|\E A_{I^c \times I}\|
\end{equation}
where the second inequality follows by \eqref{eq: EE}.
By assumption, the matrix $\E A_{I \times [n]}$ has $|I| \le 10n/d$ rows and each of its entries
is bounded by $d/n$. Hence the $\ell_1$ norm of all rows is bounded by $d$, and
the $\ell_1$ norm of all columns is bounded by $10$. Lemma~\ref{lem: norm bound with L1 condition}
implies that $\|\E A_{I \times [n]}\| \le \sqrt{10 d}$. A similar bound holds for
the second term of \eqref{eq: EA on NN EE}. This yields
$$
\|\E A_{\NN \cap \EE}\| \le 5 \sqrt{d},
$$
so we handled the second term in \eqref{eq: NN on EE}. Recalling that the first term there is
bounded by $C r^{3/2} \sqrt{d}$, we conclude that
$\|(A-A')_\NN\| \le 2C r^{3/2} \sqrt{d}$.

Returning to \eqref{eq: on CC decomposed}, we recall that the first term in the right hand
is bounded by $C r^{3/2} \sqrt{d}$, and we just bounded the second term by $2C r^{3/2} \sqrt{d}$.
Hence
$$
\|(A' - \E A)_\NN\| \le 4C r^{3/2} \sqrt{d}.
$$

\medskip

{\bf Step 2. Deviation on $\RR$ and $\CC$.}
By triangle inequality, we have
$$
\|(A' - \E A)_\RR\| \le \|A'_\RR\| + \|\E A_\RR\|.
$$
Recall that $0 \le A_\RR' \le A_\RR$ entrywise.
By Theorem~\ref{thm: decomposition}, each of the rows of $A_\RR$, and thus also of $A_\RR'$,
has $\ell_1$ norm at most $32 r$. Moreover, by definition of $d'$,
each of the columns of $A'$, and thus also of $A_\RR'$, has $\ell_1$ norm at most $d'$.
Lemma~\ref{lem: norm bound with L1 condition} implies that $\|A'_\RR\| \le \sqrt{32r d'}$.

The matrix $\E A_\RR$ can be handled similarly. By Theorem~\ref{thm: decomposition},
it has at most $n/d$ entries in each row, and all entries are bounded by $d/n$.
Thus each column of $\E A_\RR$ has $\ell_1$ norm at most $1$, and
and each row has $\ell_1$ norm at most $d$. Lemma~\ref{lem: norm bound with L1 condition}
implies that $\|\E A_\RR\| \le \sqrt{d}$.

We showed that
$$
\|(A' - \E A)_\RR\| \le \sqrt{32r d'} + \sqrt{d}.
$$
A similar bound holds for $\|(A' - \E A)_\CC\|$. Combining the bounds on the deviation of
$A'-\E A$ on $\NN$, $\RR$ and $\CC$ and putting them into \eqref{eq: decomposition A'-EA},
we conclude that
$$
\|A' - \E A\| \le 4C r^{3/2} \sqrt{d} + 2 \big( \sqrt{32r d'} + \sqrt{d} \big).
$$
Simplifying this inequality, we complete the proof of the main part of Theorem~\ref{thm: main formal}.
\end{proof}

\section{Proof of Decomposition Theorem~\ref{thm: decomposition}}			\label{s: decomposition proof}

\subsection{Outline of the argument}				\label{s: outline}

We will construct the decomposition in Theorem~\ref{thm: decomposition} by an iterative procedure.
The first and crucial step is to find a big block\footnote{In this paper, by block we mean a product set $I \times J$
  with arbitrary index subsets $I, J \subset [n]$. These subsets are not required to be intervals of successive integers.}
$\NN' \subset [n] \times [n]$ of size at least $(n-n/d) \times n/2$
on which $A$ concentrates, i.e.
$$
\|(A-\E A)_{\NN'}\| = O(\sqrt{d}).
$$
To find such block, we first establish concentration
in $\ell_\infty \to \ell_2$ norm; this
can be done by standard probabilistic techniques.
Next, we can automatically upgrade this to concentration in the spectral norm ($\ell_2 \to \ell_2$)
once we pass to an appropriate block $\NN'$. This can be done using a general result from functional
analysis, which we call Grothendieck-Pietsch factorization.

Repeating this argument for the transpose, we obtain another block $\NN''$ of size at least
$n/2 \times (n-n/d)$ where the graph concentrates as well. So the graph concentrates on
$\NN_0 := \NN' \cup \NN''$.
The ``core'' $\NN_0$ will form the first part of the class $\NN$ we are constructing.

It remains to control the graph on the complement of $\NN_0$. That set of edges is quite small; it can be described as
a union of a block $\CC_0$ with $n/d$ rows, a block $\RR_0$ with $n/d$ columns and an exceptional $n/2 \times n/2$ block; see Figure~\ref{fig: first decomposition} for illustration.
We may consider $\CC_0$ and $\RR_0$ as the first parts of the future classes $\CC$ and $\RR$ we are constructing.

Indeed, since $\CC_0$ has so few rows, the expected number of ones in each column of $\CC_0$ is bounded by $1$.
For simplicity, let us think that all columns of $\CC_0$ have $O(1)$ ones as desired.
(In the formal argument, we will add the bad columns to the exceptional block.)
Of course, the block $\RR_0$ can be handled similarly.

At this point, we decomposed $[n] \times [n]$ into $\NN_0$, $\RR_0$, $\CC_0$ and
an exceptional $n/2 \times n/2$ block. Now we repeat the process for the exceptional block,
constructing $\NN_1$, $\RR_1$, and $\CC_1$ there, and so on. Figure~\ref{fig: decomposition-iteration}	
illustrates this process.
At the end, we choose $\NN$, $\RR$ and $\CC$ to be the unions of the blocks $\NN_i$, $\RR_i$ and $\CC_i$
respectively.

\begin{figure}[htp]			
  \centering
  \begin{subfigure}[b]{0.3\textwidth}
    \raisebox{10pt}{\includegraphics[width=0.85\textwidth]{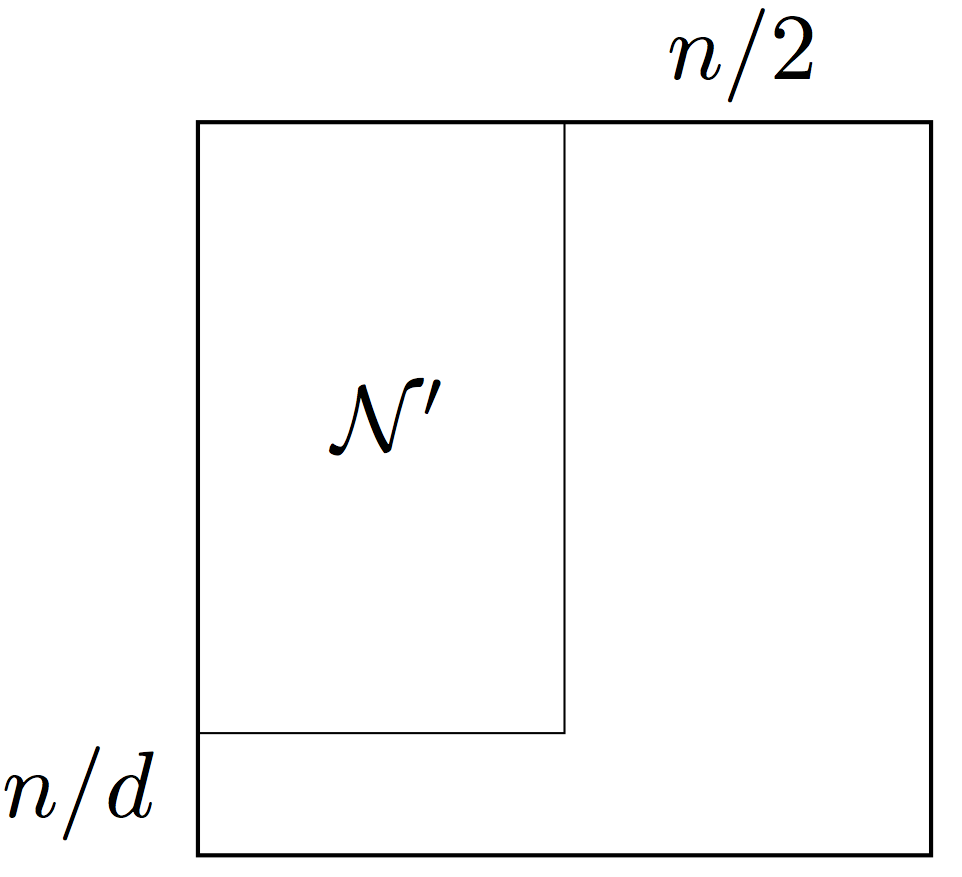}}
    \caption{The core.}
    \label{fig: first core block}
  \end{subfigure}
  \quad
  \begin{subfigure}[b]{0.3\textwidth}
    \includegraphics[width=\textwidth]{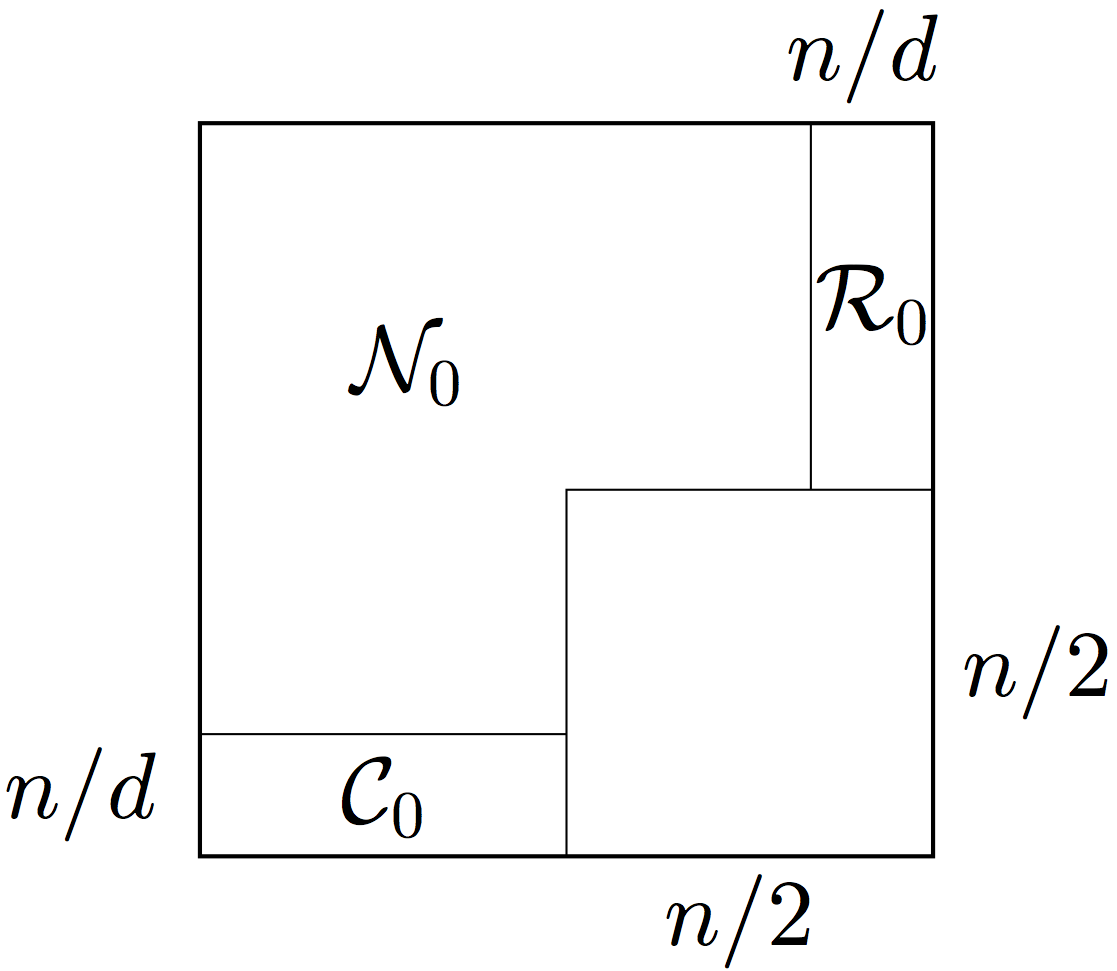}
    \caption{After the first step.}
    \label{fig: first decomposition}
  \end{subfigure}
  \quad
  \begin{subfigure}[b]{0.3\textwidth} \qquad
    \raisebox{10pt}{\includegraphics[width=0.7\textwidth]{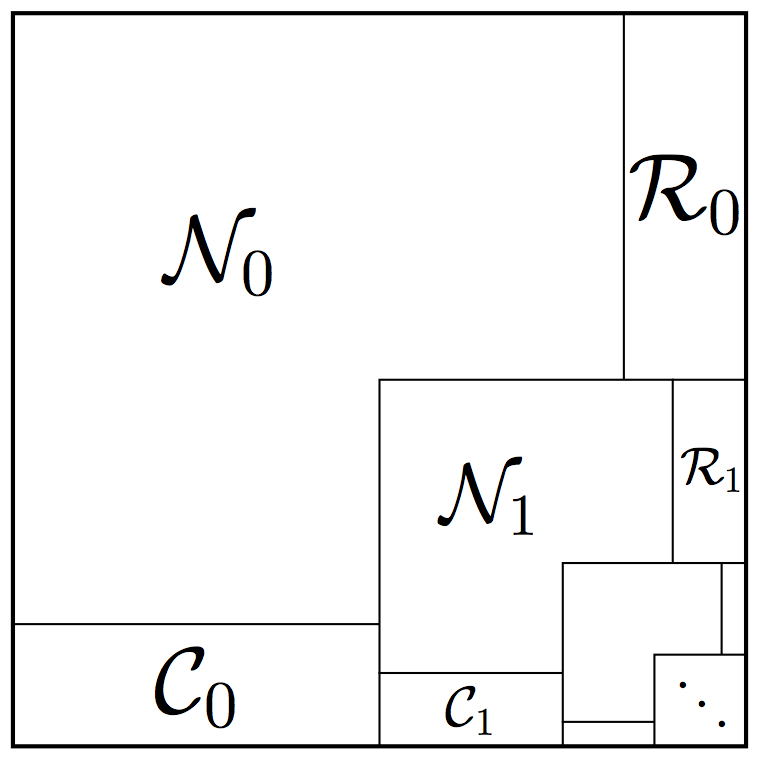}}
    \caption{Final decomposition.}
    \label{fig: decomposition-iteration}
  \end{subfigure}
  \caption{Constructing decomposition iteratively in the proof of Theorem~\ref{thm: decomposition}.}
\end{figure}

Two precautions have to be taken in this argument. First, we need to make concentration on the core blocks $\NN_i$
{\em better at each step}, so that the sum of those error bounds would not depend
of the total number of steps. This can be done with little effort, with the help of the exponential decrease of
the size of the blocks $\NN_i$.
Second, we have a control of the sizes but not locations of the exceptional blocks.
Thus to be able to carry out the decomposition argument inside
an exceptional block, we need to make the argument valid {\em uniformly} over all blocks of that size.
This will require us to be delicate with probabilistic arguments, so we can take a union
bound over such blocks.

\subsection{Grothendieck-Pietsch factorization}				\label{s: GP}

As we mentioned in the previous section, our proof of Theorem~\ref{thm: decomposition}
is based on Grothendieck-Pietsch factorization.
This general and well known result in functional analysis \cite{Pietsch1978, Pisier1986}
has already been used in a similar probabilistic context, see \cite[Proposition 15.11]{Ledoux&Talagrand1991}.

Grothendieck-Pietsch factorization compares two matrix norms,
the $\ell_2 \to \ell_2$ norm (which we call the spectral norm )
and the $\ell_\infty \to \ell_2$ norm.
For a $k \times m$ matrix $B$, these norms are defined as
$$
\|B\| = \max_{\|x\|_2 = 1} \|Bx\|_2,
\qquad
\|B\|_{\infty \to 2} = \max_{\|x\|_\infty = 1} \|Bx\|_2 = \max_{x \in \{-1,1\}^m} \|Bx\|_2.
$$
The $\ell_\infty \to \ell_2$ norm is usually easier to control, since the supremum is
taken with respect to the discrete set $\{-1,1\}^m$, and any vector there has all
coordinates of the same magnitude.

To compare the two norms, one can start with the obvious inequality
$$
\frac{\|B\|_{\infty \to 2}}{\sqrt{m}} \le \|B\| \le \|B\|_{\infty \to 2}.
$$
Both parts of this inequality are optimal, so there is an unavoidable slack between the upper
and lower bounds.
However, Grothendieck-Pietsch factorization allows us to tighten the inequality
by changing $B$ sightly. The next two results offer two ways to change $B$ --
introduce weights and pass to a sub-matrix.

\begin{theorem}[Grothendieck-Pietsch's factorization, weighted version]		\label{thm: GP factorization}
  Let $B$ be a $k \times m$ real matrix.
  Then there exist positive weights $\mu_j$ with $\sum_{j = 1}^m \mu_j =1$ such that
  \begin{equation}         \label{eq: GP factorization}
  \| B \|_{\infty\rightarrow 2} \le \| B D_\mu^{-1/2}\| \le \sqrt{ \pi/2 } \| B \|_{\infty\rightarrow 2},
  \end{equation}
  where $D_\mu = \diag(\mu_j)$ denotes the $m\times m$ diagonal matrix
  with weights $\mu_j$ on the   diagonal.
\end{theorem}

This result is a known combination of the Little Grothendieck Theorem (see \cite[Corollary~10.10]{Tomczak-Jaegermann1989} and \cite{Pisier2012}) and Pietsch Factorization (see \cite[Theorem~9.2]{Tomczak-Jaegermann1989}).
In an explicit form, a version of this result can be found e.g. in \cite[Proposition 15.11]{Ledoux&Talagrand1991}.
The weights $\mu_j$ can be computed algorithmically, see \cite{Tropp2009}.

The following related version of Grothendieck-Pietsch's factorization can be especially useful in
probabilistic contexts, see \cite[Proposition 15.11]{Ledoux&Talagrand1991}.
Here and in the rest of the paper, we denote
by $B_{I \times J}$ the sub-matrix of a matrix $B$ with rows indexed by a subset $I$ and columns
indexed by a subset $J$.

\begin{theorem}[Grothendieck-Pietsch factorization, sub-matrix version]		\label{thm: GP}
  Let $B$ be a $k \times m$ real matrix and $\d>0$. Then there exists
  $J \subseteq [m]$ with $|J| \ge (1-\d) m$ such that
  $$
  \| B_{[k] \times J}\| \le \frac{2 \| B \|_{\infty\rightarrow 2}}{ \sqrt{\d m} }.
  $$
\end{theorem}

\begin{proof}
Consider the weights $\mu_i$ given by Theorem~\ref{thm: GP factorization},
and choose $J$ to consist of the indices $j$ satisfying $\mu_j \le 1/(\d m)$.
Since $\sum_j \mu_j = 1$, the set $J$ must contain at least $(1-\d) m$ indices as claimed.
Furthermore, the diagonal entries of $(D_\mu^{-1/2})_{J\times J}$
are all bounded from below by $\sqrt{ \d m}$, which yields
$$
\|(B D_\mu^{-1/2})_{[k] \times J}\| \ge \sqrt{\d m} \|B_{[k] \times J}\|.
$$
On the other hand, by \eqref{eq: GP factorization} the left-hand side of this inequality
is bounded by $2 \| B \|_{\infty\rightarrow 2}$. Rearranging the terms, we complete the proof.
\end{proof}

\subsection{Concentration on a big block}

We are starting to work toward constructing the core part $\NN$ in Theorem~\ref{thm: decomposition}.
In this section we will show how to find a big block on which the adjacency matrix $A$ concentrates.
First we will establish concentration in $\ell_\infty \to \ell_2$ norm, and then, using Grothendieck-Pietsch factorization, in the spectral norm.

The lemmas of this and next section should be best understood for $m=n$, $I = J = [n]$ and $\a=1$.
In this case, we are working with the entire adjacency matrix, and trying to make the first step
in the iterative procedure. The further steps will require us to handle smaller blocks $I \times J$;
the parameter $\a$ will then become smaller in order to achieve better concentration for smaller blocks.

\begin{lemma}[Concentration in $\ell_\infty\rightarrow \ell_2$ norm]	\label{lem: concentration in cut norm}
  Let $1 \le m \le n$ and $\a \ge m/n$.
  Then for $r \ge 1$ the following holds with probability at least $1-n^{-r}$.
  Consider a block $I \times J$ of size $m \times m$.
  Let $I'$ be the set of indices of the rows of $A_{I \times J}$ that contain at most $\a d$ ones.
  Then
  \begin{equation}\label{concentration in infinity to two norm}
    \|( A - \E A)_{I' \times J}\|_{\infty \rightarrow 2}
    \le C \sqrt{ \a d m r \log (e n/m) }.
  \end{equation}
\end{lemma}

\begin{proof}
By definition,
\begin{equation}\label{block error in infinity to two norm}
\|( A - \E A)_{I' \times J}\|_{\infty \to 2}^2
= \max_{x \in \{-1,1\}^m} \sum_{i\in I'} \Big( \sum_{j \in J} (A_{ij} - \E A_{ij}) x_j \Big)^2
= \max_{x \in \{-1,1\}^m } \sum_{i\in I} \big( X_i \xi_i \big)^2
\end{equation}
where we denoted
$$
X_i := \sum_{j \in J} (A_{ij} - \E A_{ij}) x_j, \qquad
\xi_i := \one_{ \left\{ \sum_{i \in J} A_{ij} \le \a d \right\} }.
$$

Let us first fix a block $I \times J$ and a vector $x \in \{-1,1\}^m$. Let us analyze the
independent random variables $X_i \xi_i$.
Since $|X_i| \le \sum_{j \in J} |A_{ij} - \E A_{ij}| \le \sum_{j \in J} A_{ij}$, it follows by
definition of $\xi_i$ that
\begin{equation}         \label{eq: Xi xi bounded}
|X_i \xi_i| \le \a d.
\end{equation}

A more useful bond will follow from Bernstein's inequality.
Indeed, $X_i$ is a sum of $m$ independent random variables
with zero means and variances at most $d/n$.
By Bernstein's inequality, for any $t>0$ we have
\begin{equation}\label{bound on column sum of error}
  \Pr{ |X_i \xi_i| > t m }
  \le \Pr{ |X_i| > t m }
  \le 2 \exp \left( \frac{ - m t^2 / 2 }{ d/n + t/3 } \right), \quad t \ge 0.
\end{equation}
For $tm \le \a d$, this can be further bounded by $2\mathrm{exp}(-m^2 t^2 / 4 \a d)$,
once we use the assumption $\a \ge m/n$.
For $tm > \a d$, the left-hand side of \eqref{bound on column sum of error} is
automatically zero by \eqref{eq: Xi xi bounded}.
Therefore
\begin{equation}\label{subgaussian bound}
  \Pr{ |X_i \xi_i| > t m }
  \le 2 \ \mathrm{ exp} \left( \frac{ - m^2 t^2}{ 4 \a d} \right), \qquad t \ge 0.
\end{equation}

We are now ready to bound the right-hand side of \eqref{block error in infinity to two norm}.
By \eqref{subgaussian bound}, the random variable $X_i \xi_i$ is sub-gaussian\footnote{For
  definitions and basic facts about sub-gaussian random variables, see e.g. \cite{v-rmt-tutorial}.}
with sub-gaussian norm at most $\sqrt{\a d}$.
It follows that $(X_i \xi_i)^2$ is sub-exponential with sub-exponential norm
at most $2 \a d$.
Using Bernstein's inequality for sub-exponential random variables
(see Corrollary~5.17 in \cite{v-rmt-tutorial}), we have
\begin{equation}\label{eq: sum (X_ixi_i) square}
  \Pr{ \sum_{i \in I} \big(X_i \xi_i \big)^2 > \e m \a d }
  \le 2 \ \mathrm{exp} \left[ - c \min \left( \e^2 , \e \right)m\right], \qquad \e \ge 0.
\end{equation}
Choosing $\e := (10/c) r \log(en/m)$, we bound this probability by $(en/m)^{-5rm}$.

Summarizing, we have proved that for fixed $I,J \subseteq [n]$ and $x \in \{-1,1\}^m$,
with probability at least $1-(en/m)^{-5rm}$, the following holds:
\begin{equation}\label{final bound on sum of exponential}
  \sum_{i \in I} \big(X_i \xi_i \big)^2 \le (10/c) r \log(en/m) \cdot m \a d.
\end{equation}
Taking a union bound over all possibilities of $m,I,J,x$ and using
\eqref{block error in infinity to two norm}, \eqref{final bound on sum of exponential},
we see that the conclusion of the lemma holds with probability at least
\begin{equation*}
  1 - \sum_{m=1}^n 2^m \binom{n}{m}^2  \left( \frac{en}{m} \right)^{-5rm}
  \ge 1 - n^{-r}.
\end{equation*}
The proof is complete.
\end{proof}

Applying Lemma~\ref{lem: concentration in cut norm}
followed by Grothendieck-Piesch factorization (Theorem~\ref{thm: GP}),
we obtain the following.

\begin{lemma}[Concentration in spectral norm]    \label{lem: concentration in spectral norm}
  Let $1 \le m \le n$ and $\a \ge m/n$.
  Then for $r \ge 1$ the following holds with probability at least $1-n^{-r}$.
  Consider a block $I \times J$ of size $m \times m$.
  Let $I'$ be the set of indices of the rows of $A_{I \times J}$ that contain at most $\a d$ ones.
  Then one can find a subset $J^\prime \subseteq J$ of at least $3m/4$ columns such that
  \begin{equation}\label{eq: concentration in spectral norm}
    \|( A - \E A)_{I^\prime \times J^\prime}\|
    \le C \sqrt{ \a d r \log (e n/m) }.
  \end{equation}
\end{lemma}

\subsection{Restricted degrees}

The two simple lemmas of this section will help us to handle the
part of the adjacency matrix outside the core block constructed in
Lemma~\ref{lem: concentration in spectral norm}.
First, we show that almost all rows have at most $O(\a d)$ ones,
and thus are included in the core block.

\begin{lemma}[Degrees of subgraphs]			\label{lem: degrees of subgraphs}
  Let $1 \le m \le n$ and $\a \ge \sqrt{m/n}$.
  Then for $r \ge 1$ the following holds with probability at least $1-n^{-r}$.
  Consider a block $I \times J$ of size $m \times m$.
  Then all but $m/\a d$ rows of $A_{I \times J}$ have at most $8r\a d$ ones.
\end{lemma}

\begin{proof}
Fix a block $I \times J$, and
denote by $d_i$ the number of ones in the $i$-th row of $A_{I \times J}$.
Then $\E d_i \leq md/n$ by the assumption.
Using Chernoff's inequality, we obtain
$$
\Pr{d_i> 8r \a d} \le \Big( \frac{ 8r \a d }{ emd/n } \Big) ^ {-8r\a d}
\le \Big( \frac{ 2\a n }{ m } \Big)^{-8r\a d} =: p.
$$
Let $S$ be the number of rows $i$ with $d_i > 8r \a d$.
Then $S$ is a sum of $m$ independent Bernoulli random variables
with expectations at most $p$.
Again, Chernoff's inequality implies
$$
\Pr{S > m/ \a d}
\le (e p \a d)^ {m/\a d}
\le p^{m / 2 \a d} = \Big( \frac{ 2\a n }{ m } \Big)^{-4rm}.
$$
The second inequality here holds since $e \a d \le p^{-1/2}$.
(To see this, notice that the definition of $p$ and assumption on $\a$ imply that
$p^{-1/2} = (2 \a n / m)^{4 r \a d} \ge 2^{4 r \a d}$.)

It remains to take a union bound over all blocks $I \times J$.
We obtain that the conclusion of the lemma holds with probability at least
$$
1-\sum_{m=1}^n \binom{n}{m}^2 \Big( \frac{ 2\a n }{ m } \Big)^{-4rm}
\ge 1 - n^{-r}.
$$
In the last inequality we used the assumption that $\a \ge \sqrt{m/n}$.
The proof is complete.
\end{proof}

Next, we handle the block of rows that do have too many ones.
We show that most {\em columns} of this block have $O(1)$ ones.

\begin{lemma}[More on degrees of subgraphs]	\label{lem: more degrees of subgraphs}
  Let $1 \le m \le n$ and $\a \ge \sqrt{m/n}$.
  Then for $r \ge 1$ the following holds with probability at least $1-n^{-r}$.
  Consider a block $I \times J$ of size $k \times m$ with some $k \le m / \a d$.
  Then all but $m/4$ columns of $A_{I \times J}$ have at most $32r$ ones.
\end{lemma}

\begin{proof}
Fix $I$ and $J$, and
denote by $d_j$ the number of ones in the $j$-th column of $A_{I \times J}$.
Then $\E d_j\leq kd/n \le m/ \a n$ by assumption.
Using Chernoff's inequality, we have
$$
\Pr{d_j > 32r} \le \Big( \frac{32r}{ e m / \a n} \Big)^{-32r}
\le \Big( \frac{10\a n}{ m } \Big)^{-32r} =: p.
$$
Let $S$ be the number of columns $j$ with $d_j>32r$.
Then $S$ is a sum of $m$ independent Bernoulli random variables
with expectations at most $p$.
Again, Chernoff's inequality implies
$$
\Pr{S > m/4} \le (4ep)^{m/4} \le p^{m/6} \le \Big( \frac{10\a n}{ m } \Big)^{-5rm}.
$$
The second inequality here holds since $4e < p^{1/2}$, which in turn follows by assumption on $\a$.

It remains to take a union bound over all blocks $I \times J$. It is enough to consider the blocks
with largest possible number of columns, thus with $k = \lceil m/\a d \rceil$.
We obtain that the conclusion of the lemma holds with probability at least
$$
1-\sum_{m=1}^n \binom{n}{m}\binom{n}{\lceil m/ \a d \rceil} \Big( \frac{10\a n}{ m } \Big)^{-5rm}
\le 1-n^{-r}.
$$
In the last inequality we used the assumption that $\a \ge \sqrt{m/n}$.
The proof is complete.
\end{proof}

\subsection{Iterative decomposition: proof of Theorem~\ref{thm: main formal}}

Finally, we combine the tools we developed so far, and we construct an iterative
decomposition of the adjacency matrix the way we outline in Section~\ref{s: outline}.
Let us set up one step of this procedure, where we consider an $m \times m$ block
and decompose almost all of it (everything except an $m/2 \times m/2$ block)
into classes $\NN$, $\RR$ and $\CC$ satisfying the conclusion of Theorem~\ref{thm: decomposition}.
Once we can do this, we repeat the procedure for the $m/2 \times m/2$ block, etc.

\begin{lemma}[Decomposition of a block]					\label{lem: block decomposition}
  Let $1 \le m \le n$ and $\a \ge \sqrt{m/n}$.
  Then for $r \ge 1$ the following holds with probability at least $1-3n^{-r}$.
  Consider a block $I \times J$ of size $m \times m$.
  Then there exists an exceptional sub-block $I_1 \times J_1$ with dimensions at most $m/2 \times m/2$
  such that the remaining part of the block, that is $(I \times J) \setminus (I_1 \times J_1)$,
  can be decomposed into three classes $\NN$, $\RR \subset (I \setminus I_1) \times J$ and
  $\CC \subset I \times (J \setminus J_1)$ so that the following holds.
  \begin{itemize}
    \item The graph concentrates on $\NN$, namely
    $\|(A - \E A)_\NN\| \le C r^{3/2} \sqrt{\a d \log(en/m)}$.
    \item Each row of $A_\RR$ and each column of $A_\CC$ has at most $32r$ ones.
  \end{itemize}
  Moreover, $\RR$ intersects at most $n / \a d$ columns
  and $\CC$ intersects at most $n / \a d$ rows of $I \times J$.
\end{lemma}

After a permutation of rows and columns, a decomposition of the block
stated in Lemma~\ref{lem: block decomposition}
can be visualized in Figure~\ref{fig: block-decomposition-3}.

\begin{figure}[htp]			
  \centering
  \begin{subfigure}[b]{0.3\textwidth}
    \includegraphics[width=0.8\textwidth]{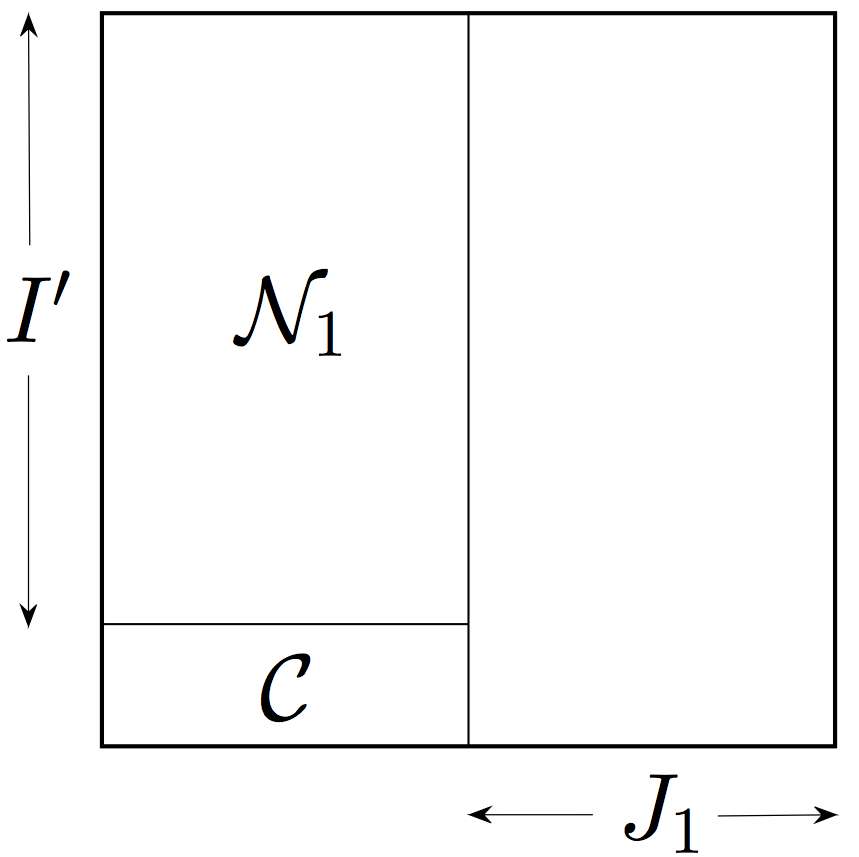}
    \caption{Initial step.}
    \label{fig: block-decomposition-1}
  \end{subfigure}
  \quad
  \begin{subfigure}[b]{0.3\textwidth} \qquad
    \raisebox{12pt}{\includegraphics[width=0.8\textwidth]{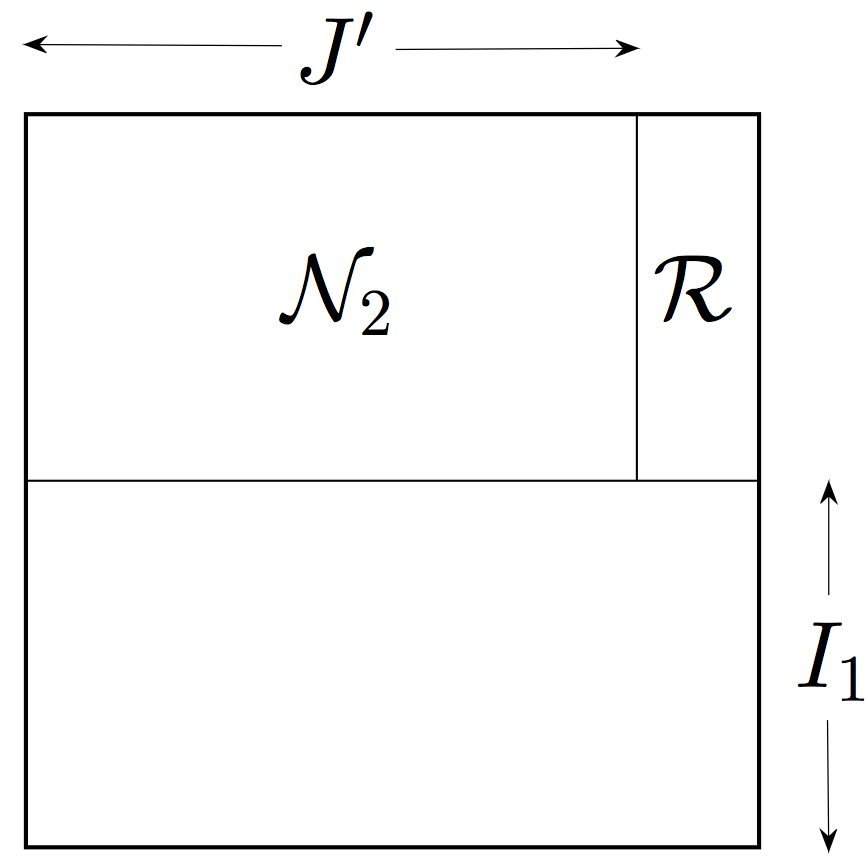}}
    \caption{Repeat for transpose.}
    \label{fig: block-decomposition-2}
  \end{subfigure}
  \quad
  \begin{subfigure}[b]{0.3\textwidth} \qquad
    \includegraphics[width=0.8\textwidth]{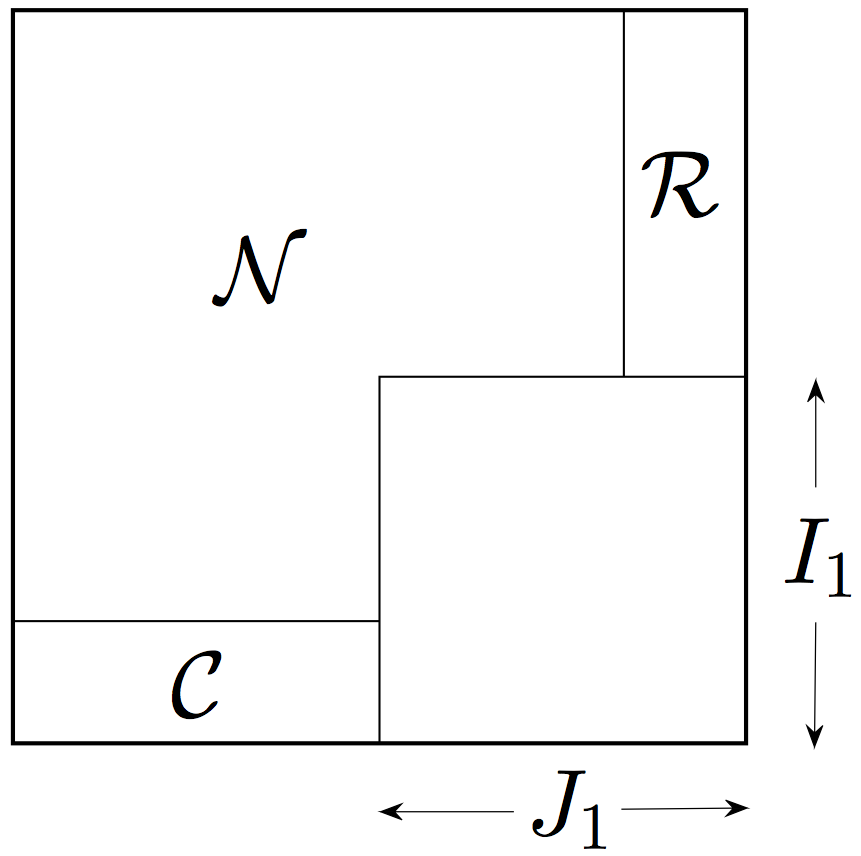}
    \caption{Final decomposition.}
    \label{fig: block-decomposition-3}
  \end{subfigure}
  \caption{Construction of a block decomposition in Lemma~\ref{lem: block decomposition}.}
  \label{fig: block-decomposition}	
\end{figure}

\begin{proof}
Since we are going to use Lemmas~\ref{lem: concentration in spectral norm}, \ref{lem: degrees of subgraphs}
and \ref{lem: more degrees of subgraphs}, let us fix realization of the random graph
that satisfies the conclusion of those three lemmas.

By Lemma~\ref{lem: degrees of subgraphs}, all but $m / \a d$ rows of $A_{I \times J}$
have at most $8r\a d$ ones; let us denote by $I'$ the set of indices of those rows with at most $8r\a d$ ones.
Then we can use Lemma~\ref{lem: concentration in spectral norm}
for the block $I' \times J$ and with $\a$ replaced by $8r\a$;
the choice of $I'$ ensures that all rows have small numbers of ones, as required in that lemma.
To control the rows outside $I'$, we may use Lemma~\ref{lem: more degrees of subgraphs}
for $(I \setminus I') \times J$; as we already noted, this block has at most $m/\a d$ rows as required
in that lemma. Intersecting the good sets of columns produced by those two lemmas,
we obtain a set of at most $m/2$ exceptional columns $J_1 \subset J$
such that the following holds.
  \begin{itemize}
    \item On the block $\NN_1 := I' \times (J \setminus J_1)$, we have
    $
    \|(A - \E A)_{\NN_1}\| \le C r^{3/2} \sqrt{\a d \log(en/m)}.
    $
    \item For the block $\CC := (I \setminus I') \times (J \setminus J_1)$,
    all columns of $A_\CC$ have at most $32r$ ones.
\end{itemize}
Figure~\ref{fig: block-decomposition-1} illustrates the decomposition of the block $I \times J$ into the set of
exceptional columns indexed by $J_1$ and good sets $\NN_1$ and $\CC$.

To finish the proof, we apply the above argument to the transpose $A^\tran$ on
the block $J \times I$. To be precise, we start with the set $J' \subset J$ of all but $m/\a d$ small columns
of $A_{I \times J}$ (those with at most $8r \a d$ ones); then we obtain an exceptional set
$I_1 \subset I$ of at most $m/2$ rows; and finally we conclude that $A$ concentrates
on the block $\NN_2 := (I \setminus I_1) \times J'$ and has small rows on the block
$\RR := (I \setminus I_1) \times (J \setminus J')$.
Figure~\ref{fig: block-decomposition-2} illustrates this decomposition.

It only remains to combine the decompositions for $A$ and $A^\tran$;
Figure~\ref{fig: block-decomposition-3} illustrates a result of the combination.
Once we define $\NN := \NN_1 \cup \NN_2$, it becomes clear that
$\NN$, $\RR$ and $\CC$ have the required properties.\footnote{It may happen that an entry
ends up in more than one class $\NN$, $\RR$ and $\CC$. In such cases, we split the tie arbitrarily.}
\end{proof}

\begin{proof}[Proof of Theorem~\ref{thm: decomposition}]
Let us fix a realization of the random graph that satisfies the conclusion of
Lemma~\ref{lem: block decomposition}.
Applying that lemma for $m=n$ and with $\a = 1$, we decompose the set of edges $[n] \times [n]$
into three classes $\NN_0$, $\CC_0$ and $\RR_0$ plus
an $n/2 \times n/2$ exceptional block $I_1 \times J_1$. Apply Lemma~\ref{lem: block decomposition}
again, this time for the block $I_1 \times J_1$, for $m=n/2$ and with $\a = \sqrt{1/2}$.
We decompose $I_1 \times J_1$ into $\NN_1$, $\CC_1$ and $\RR_1$ plus
an $n/4 \times n/4$ exceptional block $I_2 \times J_2$.

Repeat this process for $\a = \sqrt{m/n}$ where $m$ is the running size of the block;
we halve this size at each step, and so we have $\a_i \le 2^{-i/2}$.
Figure~\ref{fig: decomposition-iteration}	 illustrates a decomposition that we may obtain this way.
In a finite number of steps (actually, in $O(\log n)$ steps) the exceptional block becomes empty,
and the process terminates. At that point we have decomposed the set of edges $[n] \times [n]$
into $\NN$, $\RR$ and $\CC$, defined as the union of $\NN_i$, $\CC_i$ and $\RR_i$ respectively,
which we obtained at each step. It is clear that $\RR$ and $\CC$ satisfy the required properties.

It remains to bound the deviation of $A$ on $\NN$. By construction, $\NN_i$ satisfies
$$
\|(A - \E A)_{\NN_i}\| \le C r^{3/2} \sqrt{\a_i d \log(e\a_i)}.
$$
Thus, by triangle inequality we have
$$
\|(A - \E A)_\NN\| \le \sum_{i \ge 0} C r^{3/2} \sqrt{\a_i d \log(e\a_i)} \le C' r^{3/2} \sqrt{d}.
$$
In the second inequality we used that $\a_i \le 2^{-i/2}$, which forces the series to converge.
The proof of Theorem~\ref{thm: decomposition} is complete.
\end{proof}

\subsection{Replacing the degrees by the \texorpdfstring{$\ell_2$}{Lg} norms in Theorem~\ref{thm: main formal}}  \label{s: moreover}

Let us now prove the ``moreover'' part of Theorem~\ref{thm: main formal},
where $d'$ is the the maximal $\ell_2$ norm of the rows and columns of the
regularized adjacency matrix $A'$.
This is clearly a stronger statement than in the main part of the theorem.
Indeed, since all entries of $A'$ are bounded in absolute value by $1$,
each degree, being the $\ell_1$ norm of a row, is bounded below by the $\ell_2$ norm squared.

This strengthening is in fact easy to check. To do so, note that the definition of $d'$
was used only once in the proof of Theorem~\ref{thm: main formal}, namely in Step 2
where we bounded the norms of $A'_\RR$ and $A'_\CC$.
Thus, to obtain the strengthening, it is enough to replace the application of
Lemma~\ref{lem: norm bound with L1 condition} there by the following lemma.

\begin{lemma}
  Consider a matrix $B$ with entries in $[0,1]$.
  Suppose each row of $B$ has at most $a$ non-zero entries,
  and each column has $\ell_2$ norm at most $\sqrt{b}$.
  Then $\|B\| \le \sqrt{ab}$.
\end{lemma}

\begin{proof}
To prove the claim, let $x$ be a vector with $\|x\|_2=1$. Using Cauchy-Schwarz inequality and the assumptions, we have
\begin{align*}
\|Bx\|_2^2
  &= \sum_j \Big( \sum_i B_{ij} x_i \Big)^2
  \le \sum_j \Big( \sum_{i:\, B_{ij} \ne 0} B_{ij}^2 \sum_{i:\, B_{ij} \neq 0} x_i^2 \Big) \\
  &\le \sum_j \Big( b \sum_{i:\, B_{ij} \neq 0} x_i^2 \Big)
 = b \sum_i x_i^2 \sum_{j:\, B_{ij} \ne 0} 1
 \le b \sum_i x_i^2 a = ab.
\end{align*}
Since $x$ is arbitrary, this completes the proof.
\end{proof}

\section{Concentration of the regularized Laplacian}				\label{s: Laplacian proof}

In this section, we state the following formal version of Theorem~\ref{thm: Laplacian informal},
and we deduce it from concentration of adjacency matrices (Theorem~\ref{thm: main formal}).

\begin{theorem}[Concentration of regularized Laplacians]  \label{thm: Laplacian formal}
  Consider a random graph from the inhomogeneous Erd\"os-R\'enyi model,
  and let $d$ be as in \eqref{eq: d}.
  Choose a number $\tau>0$.
  Then, for any $r \ge 1$, with probability at least $1-e^{-r}$ one has
  $$
  \|\LL(A_\tau) - \LL(\E A_\tau)\| \le \frac{Cr^2}{\sqrt{\tau}} \Big( 1 + \frac{d}{\tau} \Big)^{5/2}.
  $$
\end{theorem}

\begin{proof}
Two sources contribute to the deviation of Laplacian -- the deviation of the adjacency matrix
and the deviation of the degrees. Let us separate and bound them individually.

\medskip

{\bf Step 1. Decomposing the deviation.}
Let us denote $\bar{A} := \E A$ for simplicity; then
$$
E := \LL(A_\tau) - \LL(\bar{A}_\tau)
= D_\tau^{-1/2} A_\tau D_\tau^{-1/2} - \bar{D}_\tau^{-1/2} \bar{A}_\tau \bar{D}_\tau^{-1/2}.
$$
Here $D_\tau = \diag(d_i + \tau)$ and $\bar{D}_\tau = \diag(\bar{d}_i + \tau)$
are the diagonal matrices with degrees of $A_\tau$ and $\bar{A}_\tau$ on the diagonal, respectively.
Using the fact that $A_\tau-\bar{A}_\tau = A - \bar{A}$, we can represent the
deviation as
$$
E = S + T
\quad \text{where} \quad
S = D_\tau^{-1/2} (A-\bar{A}) D_\tau^{-1/2}, \quad
T =D_\tau^{-1/2} \bar{A}_\tau D_\tau^{-1/2} - \bar{D}_\tau^{-1/2} \bar{A}_\tau \bar{D}_\tau^{-1/2}.
$$
Let us bound $S$ and $T$ separately.

\medskip

{\bf Step 2. Bounding $S$.} Let us introduce a diagonal matrix $\Delta$ that should be easier to work with
than $D_\tau$. Set $\Delta_{ii} = 1$ if $d_i \le 8rd$ and $\Delta_{ii} = d_i/\tau + 1$ otherwise.
Then entries of $\tau \Delta$ are upper bounded by the corresponding entries of $D_\tau$, and so
$$
\tau \|S\| \le \|\Delta^{-1/2} (A-\bar{A}) \Delta^{-1/2}\|.
$$
Next, by triangle inequality,
\begin{equation}         \label{eq: S R1 R2}
\tau \|S\| \le \| \Delta^{-1/2} A \Delta^{-1/2} - \bar{A}\| + \| \bar{A} - \Delta^{-1/2} \bar{A} \Delta^{-1/2}\|
=: R_1 + R_2.
\end{equation}

In order to bound $R_1$, we use Theorem~\ref{thm: main formal} to show that $A' := \Delta^{-1/2} A \Delta^{-1/2}$ concentrates around $\bar{A}$. This should be possible because $A'$ is
obtained from $A$ by reducing the degrees that are bigger than $8rd$.
To apply the ``moreover'' part of  Theorem~\ref{thm: main formal}, let us check the magnitude
of the $\ell_2$ norms of the rows $A_i^\prime$ of $A'$:
$$
\|A'_i\|_2^2 = \sum_{j=1}^n \frac{A_{ij}}{\Delta_{ii} \Delta_{jj}} \le \frac{d_i}{\Delta_{ii}}
\le \max(8rd, \, \tau).
$$
Here in the first inequality we used that $\Delta_{jj} \ge 1$ and $\sum_j A_{ij} = d_i$;
the second inequality follows by definition of $\Delta_{ii}$.
Applying Theorem~\ref{thm: main formal}, we obtain with probability $1-n^{-r}$ that
$$
R_1 = \|A^\prime - \bar{A}\| \le C_1 r^2 (\sqrt{d}+\sqrt{\tau}).
$$

To bound $R_2$, we note that by construction of $\Delta$, the matrices
$\bar{A}$ and $\Delta^{-1/2} \bar{A} \Delta^{-1/2}$ coincide on the block $I \times I$, where
$I$ is the set of vertices satisfying $d_i \le 8rd$. This block is very large -- indeed,
Lemma~\ref{lem: degrees of subgraphs} implies that $|I^c| \le n/d$ with probability $1-n^{-r}$.
Outside this block, i.e. on the small blocks $I^c \times [n]$ and $[n] \times I^c$,
the entries of $\bar{A} - \Delta^{-1/2} \bar{A} \Delta^{-1/2}$ are bounded by the corresponding
entries of $\bar{A}$, which are all bounded by $d/n$.
Thus, using Lemma~\ref{lem: norm bound with L1 condition}, we have
$$
R_2 \le \|\bar{A}_{I^c \times [n]}\| + \|\bar{A}_{[n] \times I^c}\| \le 2\sqrt{d}.
$$
Substituting the bounds for $R_1$ and $R_2$ into \eqref{eq: S R1 R2}, we conclude that
$$
\|S\| \le \frac{C_2 r^2}{\tau} (\sqrt{d} + \sqrt{\tau})
$$
with probability at least $1 - 2n^{-r}$.

\medskip

{\bf Step 3. Bounding $T$.}
Bounding the spectral norm by the Hilbert-Schmidt norm, we get
$$
\|T\| \le \|T\|_{\mathrm{HS}} = \sum_{i,j=1}^n T_{ij}^2,
\quad \text{where} \quad
T_{ij} = (\bar{A}_{ij}+\tau/n) \Big[ 1/\sqrt{\delta_{ij}} - 1/{\sqrt{\bar{\d}_{ij}}} \Big]
$$
and $\d_{ij} = (d_i+\tau)(d_j+\tau)$ and $\bar{\d}_{ij} = (\bar{d}_i+\tau)(\bar{d}_j+\tau)$.
To bound $T_{ij}$, we note that
$$
0 \le \bar{A}_{ij}+\tau/n \le \frac{d+\tau}{n}
\quad \text{and} \quad
\big| 1/\sqrt{\delta_{ij}} - 1/{\sqrt{\bar{\d}_{ij}}} \big|
= \left| \frac{\d_{ij} - \bar{\d}_{ij}}{\d_{ij} \sqrt{\bar{\d}_{ij}} + \bar{\d}_{ij} \sqrt{\d_{ij}}} \right|
\ge \frac{|\d_{ij} - \bar{\d}_{ij}|}{2\tau^3}.
$$
Recalling the definition of $\d_{ij}$ and $\bar{\d}_{ij}$ and
adding and subtracting $(d_i+\tau)(\bar{d}_j+\tau)$, we have
$$
\d_{ij} - \bar{\d}_{ij} = (d_i+\tau) (d_j - \bar{d}_j) + (\bar{d}_j+\tau) (d_i - \bar{d}_i).
$$
So, using the inequality $(a+b)^2 \le 2(a^2+b^2)$ and bounding $\bar{d}_j + \tau$ by $d+\tau$, we obtain
\begin{equation}         \label{eq: T2}
\|T\|^2
\le \frac{(d+\tau)^2}{n^2\tau^6}
  \Big[ \sum_{i=1}^n (d_i+\tau)^2 \sum_{j=1}^n (d_j - \bar{d}_j)^2
    + n(d+\tau)^2 \sum_{i=1}^n (d_i - \bar{d}_i)^2 \Big].
\end{equation}

We claim that
\begin{equation}         \label{eq: sum variances}
\sum_{j=1}^n (d_j - \bar{d}_j)^2 \le C_3 r^2 n d \quad \text{with probability } 1-e^{-2r}.
\end{equation}
Indeed, since the variance of each $d_i$ is bounded by $d$, the expectation of the
sum in \eqref{eq: sum variances} is bounded by $nd$. To upgrade the variance bound to an exponential
deviation bound, one can use one of the several standard methods. For example,
Bernstein's inequality implies that $X_i = d_j - \bar{d_j}$ satisfies
$\Pr{ X_i > C_4 t \sqrt{d} } \le e^{-t}$ for all $t \ge 1$. This means that the random variable $X_i^2$ belongs
to the Orlicz space $L_{\psi_{1/2}}$ and has norm $\|X_i^2\|_{\psi_{1/2}} \le C_3 d$,
see \cite{Ledoux&Talagrand1991}.
By triangle inequality, we conclude that $\|\sum_{i=1}^n X_i^2\|_{\psi_{1/2}} \le C_3 n d$,
which in turn implies \eqref{eq: sum variances}.

Furthermore, \eqref{eq: sum variances} implies
$$
\sum_{i=1}^n (d_i+\tau)^2
\le 2\sum_{i=1}^n (d_i-\bar{d}_i)^2 + 2\sum_{i=1}^n  (\bar{d}_i+\tau)^2
\le 2C_3 r^2 nd + 2n(d+\tau)^2 \le C_5 r^2 n (d+\tau)^2.
$$
Substituting this bound and \eqref{eq: sum variances} into \eqref{eq: T2} we conclude that
$$
\|T\|^2 \le \frac{(d+\tau)^2}{n^2\tau^6} \cdot C_3 r^2 nd \, \Big[ C_5 r^2 n (d+\tau)^2 + n(d+\tau)^2 \Big]
\le \frac{C_6 r^4}{\tau} \Big( 1 + \frac{d}{\tau} \Big)^5.
$$
It remains to substitute the bounds for $S$ and $T$ into the inequality
$\|E\| \le \|S\|+\|T\|$, and simplify the expression. The resulting bound holds with probability
at least $1-n^{-r}-n^{-r}-e^{-2r} \ge 1-e^{-r}$, as claimed.
\end{proof}

\section{Further questions}				\label{s: questions}

\subsection{Optimal regularization?}

The main point of our paper was that regularization helps sparse graphs to concentrate.
We have discussed several kinds of regularization in Section~\ref{s: partial cases} and
mentioned some more in Section~\ref{s: partial cases}. We found that any meaningful
regularization works, as long as it reduces the too high degrees and increases the too low degrees.
Is there an optimal way to regularize a graph? Designing the best ``preprocessing''
of sparse graphs for spectral algorithms is especially interesting
from the applied perspective, i.e. for real world networks.

On the theoretical level, can regularization of sparse graphs produce the same optimal bound
$2 \sqrt{d}(1+o(1))$ that we saw for dense graphs in \eqref{eq: concentration dense ER}?
Thus, an ideal regularization should bring all parasitic outliers of the spectrum into the bulk.
If so, this would lead to a potentially simple spectral clustering algorithm for community detection
in networks which matches the theoretical lower bounds. Algorithms with optimal rates exist for this problem
\cite{Mossel&Neeman&Sly2014, Massoulie:2014}, but their analysis is very technical.

\subsection{How exactly concentration depends on regularization?}
It would be interesting to determine how exactly the concentration
of Laplacian depends on the regularization parameter $\tau$. The dependence
in Theorem~\ref{thm: Laplacian formal} is not optimal, and we have not made efforts to
improve it. Although it is natural to choose $\tau \sim d$ as in Theorem~\ref{thm: Laplacian informal},
choosing $\tau \gg d$ could also be useful \cite{Joseph&Yu2013}.
Choosing $\tau \ll d$ may be interesting as well, for then $\LL(\E A_\tau) \approx \LL(\E A)$
and we obtain the concentration of $\LL(A_\tau)$ around the Laplacian of the expectation of the original
(rather than regularized) matrix $\E A$.

\subsection{Average expected degree?}

Both concentration results of this paper, Theorems~\ref{thm: main informal} and \ref{thm: Laplacian informal},
depend on $d = \max_{ij} np_{ij}$. Would it be possible to reduce $d$ to the maximal expected degree
$d_{ave} = \max_i \sum_j p_{ij}$?

\subsection{From random graphs to random matrices?}			

Adjacency matrices of random graphs are particular examples of random matrices.
Does the phenomenon we described, namely that regularization leads to concentration,
apply for general random matrices?
Guided by Theorem~\ref{thm: main informal}, we might expect
the following for a broader general class of random matrices $B$ with mean zero independent entries.
First, the only reason the spectral norm of $B$ is too large (and that it is determined by outliers of spectrum)
could be the existence of a large row or column. Furthermore, it might be possible to reduce the norm of
$B$ (and thus bring the outliers into the bulk of spectrum) by regularizing in some way
the rows and columns that are too large.
For related questions in random matrix theory, see the recent work \cite{Bandeira&Handel2014, Handel2015}.

\bibliography{allref}

\begin{thebibliography}{10}

\bibitem{Abbe&Bandeira&Hall2014}
E.~Abbe, A.~S. Bandeira, and G.~Hall.
\newblock Exact recovery in the stochastic block model.
\newblock {\em IEEE Transactions on Information Theory}, 62(1):471--487, 2016.

\bibitem{Alon&Kahale1997}
N.~Alon and N.~Kahale.
\newblock A spectral technique for coloring random 3-colorable graphs.
\newblock {\em SIAM J. Comput.}, (26):1733--1748, 1997.

\bibitem{amini2013pseudo}
A.~A. Amini, A.~Chen, P.~J. Bickel, and E.~Levina.
\newblock Pseudo-likelihood methods for community detection in large sparse
  networks.
\newblock {\em The Annals of Statistics}, 41(4):2097--2122, 2013.

\bibitem{Bandeira&Handel2014}
A.~Bandeira and R.~V. Handel.
\newblock Sharp nonasymptotic bounds on the norm of random matrices with
  independent entries.
\newblock {\em Annals of Probability, to appear}, 2014.

\bibitem{Bhatia1996}
R.~Bhatia.
\newblock {\em Matrix Analysis}.
\newblock Springer-Verlag New York, 1996.

\bibitem{Bickel&Chen2009}
P.~J. Bickel and A.~Chen.
\newblock A nonparametric view of network models and {N}ewman-{G}irvan and
  other modularities.
\newblock {\em Proc. Natl. Acad. Sci. USA}, 106:21068--21073, 2009.

\bibitem{Bollobas2007}
B.~Bollobas, S.~Janson, and O.~Riordan.
\newblock The phase transition in inhomogeneous random graphs.
\newblock {\em Random Structures and Algorithms}, 31:3--122, 2007.

\bibitem{Bordenave.et.al2015non-backtracking}
C.~Bordenave, M.~Lelarge, and L.~Massouli\'{e}.
\newblock Non-backtracking spectrum of random graphs: community detection and
  non-regular {R}amanujan graphs.
\newblock {\em arxiv:1501.06087}, 2015.

\bibitem{Boucheron&Lugosi&Massart2013}
S.~Boucheron, G.~Lugosi, and P.~Massart.
\newblock {\em Concentration inequalities: a nonasymptotic theory of
  independence}.
\newblock Oxford University Press, 2013.

\bibitem{Cai&Li2015}
T.~Cai and X.~Li.
\newblock Robust and computationally feasible community detection in the
  presence of arbitrary outlier nodes.
\newblock {\em Ann. Statist.}, 43(3):1027--1059, 2015.

\bibitem{Chaudhuri&Chung&Tsiatas2012}
K.~Chaudhuri, F.~Chung, and A.~Tsiatas.
\newblock Spectral clustering of graphs with general degrees in the extended
  planted partition model.
\newblock {\em Journal of Machine Learning Research Workshop and Conference
  Proceedings}, 23:35.1 -- 35.23, 2012.

\bibitem{Chin&Rao&Vu2015}
P.~Chin, A.~Rao, and V.~Vu.
\newblock Stochastic block model and community detection in the sparse graphs :
  A spectral algorithm with optimal rate of recovery.
\newblock {\em arXiv:1501.05021}, 2015.

\bibitem{ChungFan1997}
F.~R.~K. Chung.
\newblock {\em Spectral Graph Theory}.
\newblock CBMS Regional Conference Series in Mathematics, 1997.

\bibitem{Decelle.et.al.2011}
A.~Decelle, F.~Krzakala, C.~Moore, and L.~Zdeborov\'{a}.
\newblock Asymptotic analysis of the stochastic block model for modular
  networks and its algorithmic applications.
\newblock {\em Physical Review E}, 84:066106, 2011.

\bibitem{FeiOfe05}
U.~Feige and .~Ofek.
\newblock Spectral techniques applied to sparse random graphs.
\newblock {\em Wiley InterScience}, 2005.

\bibitem{FurKom80}
Z.~Füredi and J.~Komlós.
\newblock The eigenvalues of random symmetric matrices.
\newblock {\em Combinatorica}, 1:3:233--241, 1980.

\bibitem{Friedman&Kahn&Szemeredi1989}
J.~Friedman, J.~Kahn, and E.~Szemeredi.
\newblock On the second eigenvalue in random regular graphs.
\newblock {\em Proc Twenty First Annu ACMSymp Theory of Computing}, pages
  587--598, 1989.

\bibitem{Gao&Ma&Zhang&Zhou2015}
C.~Gao, Z.~Ma, A.~Y. Zhang, and H.~H. Zhou.
\newblock Achieving optimal misclassification proportion in stochastic block
  model.
\newblock {\em arXiv:1505.03772}, 2015.

\bibitem{Guedon&Vershynin2014}
O.~Gu\'{e}don and R.~Vershynin.
\newblock Community detection in sparse networks via grothendieck's inequality.
\newblock {\em Probability Theory and Related Fields, to appear}, 2014.

\bibitem{Hajek&Wu&Xu2014}
B.~Hajek, Y.~Wu, and J.~Xu.
\newblock Achieving exact cluster recovery threshold via semidefinite
  programming.
\newblock {\em arXiv:1412.6156}, 2014.

\bibitem{Handel2015}
R.~V. Handel.
\newblock On the spectral norm of inhomogeneous random matrices.
\newblock {\em arXiv:1502.05003}, 2015.

\bibitem{Holland83}
P.~W. Holland, K.~B. Laskey, and S.~Leinhardt.
\newblock Stochastic blockmodels: first steps.
\newblock {\em Social Networks}, 5(2):109--137, 1983.

\bibitem{Joseph&Yu2013}
A.~Joseph and B.~Yu.
\newblock Impact of regularization on spectral clustering.
\newblock {\em Ann. Statist.}, 44(4):1765--1791, 2016.

\bibitem{Krivelevich&Sudakov2003}
M.~Krivelevich and B.~Sudakov.
\newblock The largest eigenvalue of sparse random graphs.
\newblock {\em Combin Probab Comput}, 12:61--72, 2003.

\bibitem{ledoux2001}
M.~Ledoux.
\newblock {\em The Concentration of Measure Phenomenon}, volume~89 of {\em
  Mathematical Surveys and Monographs}.
\newblock Amer. Math. Society, 2001.

\bibitem{Ledoux&Talagrand1991}
M.~Ledoux and M.~Talagrand.
\newblock {\em Probability in Banach spaces: Isoperimetry and processes}.
\newblock Springer-Verlag, Berlin, 1991.

\bibitem{Lei&Rinaldo2013}
J.~Lei and A.~Rinaldo.
\newblock Consistency of spectral clustering in stochastic block models.
\newblock {\em Ann. Statist.}, 43(1):215--237, 2015.

\bibitem{Lu&Peng2013}
L.~Lu and X.~Peng.
\newblock Spectra of edge-independent random graphs.
\newblock {\em The electronic journal of combinatorics}, 20(4), 2013.

\bibitem{Massoulie:2014}
L.~Massouli{\'e}.
\newblock Community detection thresholds and the weak {R}amanujan property.
\newblock In {\em Proceedings of the 46th Annual ACM Symposium on Theory of
  Computing}, STOC '14, pages 694--703, 2014.

\bibitem{McS01}
McSherry.
\newblock Spectral partitioning of random graphs.
\newblock {\em Proc. 42nd FOCS}, pages 529--537, 2001.

\bibitem{Montanari&Sen2015}
A.~Montanari and S.~Sen.
\newblock Semidefinite programs on sparse random graphs and their application
  to community detection.
\newblock {\em arXiv:1504.05910}, 2015.

\bibitem{Mossel&Neeman&SlyOnConsistencyThresholds2014}
E.~Mossel, J.~Neeman, and A.~Sly.
\newblock Consistency thresholds for binary symmetric block models.
\newblock {\em arXiv:1407.1591}, 2014.

\bibitem{Mossel&Neeman&Sly2014}
E.~Mossel, J.~Neeman, and A.~Sly.
\newblock A proof of the block model threshold conjecture.
\newblock {\em arXiv:1311.4115}, 2014.

\bibitem{Mossel&Neeman&Sly2014a}
E.~Mossel, J.~Neeman, and A.~Sly.
\newblock Reconstruction and estimation in the planted partition model.
\newblock {\em Probability Theory and Related Fields}, 2014.

\bibitem{Oliveira2010}
R.~Oliveira.
\newblock Concentration of the adjacency matrix and of the laplacian in random
  graphs with independent edges.
\newblock {\em arXiv:0911.0600}, 2010.

\bibitem{Pietsch1978}
A.~Pietsch.
\newblock {\em Operator Ideals}.
\newblock North-Holland Amsterdam, 1978.

\bibitem{Pisier1986}
G.~Pisier.
\newblock {\em Factorization of linear operators and geometry of Banach
  spaces}.
\newblock Number 60 in CBMS Regional Conference Series in Mathematics. AMS,
  Providence, 1986.

\bibitem{Pisier2012}
G.~Pisier.
\newblock Grothendieck’s theorem, past and present.
\newblock {\em Bulletin (New Series) of the American Mathematical Society},
  49(2):237--323, 2012.

\bibitem{Qin&Rohe2013}
T.~Qin and K.~Rohe.
\newblock Regularized spectral clustering under the degree-corrected stochastic
  blockmodel.
\newblock In {\em Advances in Neural Information Processing Systems}, pages
  3120--3128, 2013.

\bibitem{Stein&Shakarchi2011}
E.~M. Stein and R.~Shakarchi.
\newblock {\em Functional Analysis: Introduction to Further Topics in
  Analysis}.
\newblock Princeton University Press, 2011.

\bibitem{Tomczak-Jaegermann1989}
N.~Tomczak-Jaegermann.
\newblock {\em Banach-Mazur distances and finite-dimensional operator ideals}.
\newblock John Wiley \& Sons, Inc., New York, 1989.

\bibitem{Tropp2009}
J.~A. Tropp.
\newblock Column subset selection, matrix factorization, and eigenvalue
  optimization.
\newblock {\em Proceedings of the Twentieth Annual ACM-SIAM Symposium on
  Discrete Algorithms}, pages 978--986, 2009.

\bibitem{v-rmt-tutorial}
R.~Vershynin.
\newblock Introduction to the non-asymptotic analysis of random matrices.
\newblock In Y.~Eldar and G.~Kutyniok, editors, {\em Compressed sensing: theory
  and applications}. Cambridge University Press.
\newblock Submitted.

\bibitem{Vu2007}
V.~Vu.
\newblock Spectral norm of random matrices.
\newblock {\em Combinatorica}, 27(6):721--736, 2007.

\end{thebibliography}
\bibliographystyle{abbrv}

\end{document}